\documentclass[12pt]{article}
\usepackage{authblk}
\usepackage{bbm}
\usepackage{url}
\usepackage{times}
\usepackage[colorlinks=true, citecolor=blue]{hyperref}

\usepackage{framed}
\usepackage[normalem]{ulem}
\usepackage{mathtools}
\usepackage{nomencl}
\makenomenclature

\usepackage[utf8]{inputenc}
\usepackage{scalerel}
\usepackage{enumerate}
\usepackage{authblk}
\usepackage{wrapfig}
\usepackage{tikz-cd}
\usepackage[new]{old-arrows}
\usepackage{amsmath,amscd}
\usepackage{amsthm}
\usepackage{amsfonts}
\usepackage{color}
\usepackage{thmtools}
\usepackage{thm-restate}
\usepackage{amssymb}
\usepackage[margin=1in]{geometry}
\usepackage{stackrel}
\usepackage{mathtools}

\newtheorem{theorem}{Theorem}[section]
\newtheorem{proposition}[theorem]{Proposition}
\newtheorem{corollary}[theorem]{Corollary}
\newtheorem{lemma}[theorem]{Lemma}

\newtheorem{remark}[theorem]{Remark}

\theoremstyle{definition}
\newtheorem{definition}{Definition}
\newtheorem{example}[theorem]{Example}

\newcommand{\R}{\mathbb{R}}

\newcommand{\Sp}{\mathbb{S}}
\newcommand{\diam}{\mathrm{diam}}
\newcommand{\dgh}{d_\mathrm{GH}}
\newcommand{\dH}{d_\mathrm{H}}

\newcommand{\mr}{\mathrm{MR}}

\newcommand{\TS}{\mathbf{E}}
\newcommand{\comax}{\mathrm{comax}}

\newcommand{\vr}{\mathrm{VR}}

\newcommand{\ds}{d_{\Sp^1}}
\newcommand{\norm}[1]{\left\lVert#1\right\rVert}
\newcommand{\eps}{\varepsilon}

\usepackage{tikz,tikz-3dplot}
\newcommand{\caE}{\mathcal{E}}
\newcommand{\caF}{\mathcal{F}}
\newcommand{\bbD}{\mathbb{D}}

\newcommand{\lbracket}{\left( }
\newcommand{\rbracket}{\right) }

\newcounter{claimcount}
\AtBeginEnvironment{proof}{\setcounter{claimcount}{0}}

\AtBeginEnvironment{proof}{\setcounter{claimcount}{0}}

\definecolor{pink}{RGB}{255,20,147}
\definecolor{darkblue}{rgb}{0.0, 0.0, 0.8}
\definecolor{darkred}{rgb}{0.8, 0.0, 0.0}
\definecolor{darkgreen}{rgb}{0.0, 0.8, 0.0}
\definecolor{purple}{RGB}{153,50,204}

\begin{document}
\title{Some results about the Tight Span of spheres}

\author[1]{Sunhyuk Lim}
\author[2]{Facundo M\'emoli}
\author[3]{Zhengchao Wan}
\author[4]{Qingsong Wang}
\author[5]{Ling Zhou}

\affil[1]{Max Planck Institute for Mathematics in the Sciences, Leipzig\\ \texttt{sulim@mis.mpg.de}}
\affil[2]{Department of Mathematics and Department of Computer Science and Engineering, The Ohio State University\\ \texttt{memoli@math.osu.edu}}
\affil[3]{Department of Mathematics, The Ohio State University\\ \texttt{wan.252@osu.edu}}
\affil[4]{Department of Mathematics, The Ohio State University\\ \texttt{wang.8973@osu.edu}}
\affil[5]{Department of Mathematics, The Ohio State University\\ \texttt{zhou.2568@osu.edu}}

\date{\today}
\maketitle
	
\begin{abstract}
The smallest hyperconvex metric space containing a given metric space $X$ is called the tight span of $X$. It is known that tight spans have many nice geometric and topological properties, and they are gradually becoming a target of research of both the metric geometry community and the topological/geometric data analysis community. In this paper, we study the tight span of $n$-spheres (with either geodesic metric or $\ell_\infty$-metric).
\end{abstract}

\newpage
\setlength{\nomlabelwidth}{4em}
\mbox{}
\nomenclature{\(\Sp^n\)}{$n$-dimensional sphere with its geodesic metric.}

\nomenclature{\(\Sp^n_\infty\)}{$n$-dimensional sphere with $\ell^\infty$-norm (coming from usual embedding into $\R^{n+1}$).}

\nomenclature{\(\TS(X)\)}{Tight span of the metric space $X$.}

\nomenclature{\(\diam(X)\)}{Diameter of the metric space $X$.}

\nomenclature{\(\dgh\)}{Gromov-Hausdorff distance.}

\nomenclature{\(\dH\)}{Hausdorff distance.}

\nomenclature{\(\bar{x}\)}{Antipodal point of $x$ in an antipodal metric space $X$.}

\nomenclature{\(\ell^2\)}{Hilbert space of infinite real sequences with finite square sums.}

\nomenclature{\(L^\infty(X)\)}{Banach space of real valued functions on $X$ with $\ell^\infty$-norm.}

\nomenclature{\(\mathfrak{H}\)}{Hilbert cube.}

\nomenclature{\(\R^n_\infty\)}{$\R^n$ with $\ell^\infty$-norm.}

\nomenclature{\(B_r(X,E)\)}{Open $r$-thickening of $X$ in $E$, for $X$ a sub-metric space of $E$.}

\printnomenclature
\tableofcontents

\section{Introduction}

A metric space $(E,d_E)$ is called \emph{hyperconvex} if it satisfies the following property:

\begin{framed}
For every family $(x_i,r_i)_{i \in I}$ of $x_i$ in $E$ and $r_i \geq 0$ such that $d_E(x_i,x_j) \leq r_i + r_j$ for every $i,j$ in $I$, there exists a point $x\in E$ such that $d_E(x_i,x) \leq r_i$ for every $i$ in $I$.
\end{framed}

Hyperconvex spaces were first studied by Aronszajn and Panitchpakdi in \cite{aronszajn1956extension}, and the authors proved that a metric space $E$ is hyperconvex if and only if $E$ satisfies the following ``injectivity" condition:

\begin{framed}
For each $1$-Lipschitz map $f: X \to E$ and an isometric embedding of $X$ into $\tilde{X}$, there exists a $1$-Lipschitz map $\tilde{f}: \tilde{X} \to E$ extending $f$:
$$\begin{tikzcd}
X \arrow[r, hook] \arrow[dr, "f" , rightarrow]
& \tilde{X} \arrow[d, "\tilde{f}"]\\
& E
\end{tikzcd}$$
\end{framed}

This notion of injectivity of metric spaces was rediscovered in \cite{isbell1964six} and the theory was developed further in \cite{dress1984trees,lang2013injective,lang2013metric}. One important fact is that any arbitrary metric space $(X,d_X)$ can be isometrically embedded into an injective metric space. This can be easily shown by using the \emph{Kuratowski embedding} $x\longmapsto d_X(x,\cdot)$ and the fact that $L^\infty(X)$, the space of all real-valued maps on $X$ with the uniform norm, is hyperconvex.

In the first place, we became interested in hyperconvex metric spaces because of their application to Topological Data Analysis. In $\cite{lim2020vietoris}$, the authors proved that the filtration obtained through increasing thickenings of $X$ inside any hyperconvex metric space is homotopy equivalent to the Vietoris-Rips filtration. Also, the recent papers \cite{joharinad2019topology,joharinad2020topological} introduced a new notion of curvature for metric spaces through a quantification of their deviation from hyperconvexity, and suggested applications of this notion of curvature for topological/geometric data analysis.

For an arbitrary metric space $X$,  $\TS(X)$, the tight span of $X$ is defined to be the smallest (up to isometric embedding) hyperconvex space containing $X$. It is known that tight spans have many nice geometric properties. For example, $\TS(X)$ is always a contractible geodesic metric space, and it inherits the compactness, diameter, and $\delta$-hyperbolicity from its inducing space $X$ \cite{isbell1964six,lang2013injective}.

It is known that every complete metric tree $T$ is isometric to its own tight span $\TS(T)$ (hence it is itself hyperconvex) \cite[Theorem 8]{dress1984trees}. A natural next step is to understand the tight span of metric graphs: one motivation for this is that, firstly, metric graphs can approximate arbitrary geodesic metric spaces \cite{burago2001course,memoli2018metric} in the Gromov-Hausdorff sense and, secondly, the tight span is itself Gromov-Hausdorff stable \cite[Theorem 3.1]{lang2013metric}. Since the simplest example of a non-tree metric graph is $\Sp^1$ (with geodesic metric), this motivates one of the main goals of this paper which is to characterize $\TS(\Sp^1)$. Some of the arguments we use also help us obtain a better understanding of $\TS(\Sp^n)$, the tight span of $n$-spheres for $n\geq 2$. 

Another reason why precise knowledge about the structure of the tight span of spheres is important comes from applied algebraic topology (AAT) \cite{carlsson2009topology}, where the Vietoris-Rips complex (and the filtration it induces) plays a fundamental role \cite{adamaszek2017vietoris,adamaszek2018metric}. Being natural ``model spaces", it is of interest to fully characterize the persistent homology induced by the Vietoris-Rips filtrations of spheres (endowed with their geodesic metric). Since
it is known \cite[Theorem 5]{lim2020vietoris} that for any compact metric space $X$, its Vietoris-Rips filtration and the filtration $B_\bullet(X,\TS(X))$ (arising from thickening $X$ inside its tight span) are naturally isomorphic, we hope that by better understanding $\TS(\Sp^n)$ we will be able to eventually characterize the successive homotopy types of Vietoris-Rips complexes of $\Sp^n$.  For partial results in this direction, see \S\ref{sec:TSS1:httypes}.

\paragraph{Contributions and organization.}
In \S\ref{sec:prelim}, we review preliminary notions and theorems which will be required throughout this paper. We give a succinct proof of the equivalence between hyperconvexity and injectivity (Proposition \ref{prop:hypinj}). Also, with the aid of the tight span, we classify the successive homotopy types of the Vietoris-Rips filtration of tree-like metric spaces  (Corollary \ref{cor:httype_of_tree_metric_space}).\vspace{\baselineskip}

In \S\ref{sec:TSS1}, first, motivated by the cases of finite antipodal metric spaces, we prove that for the circle $\Sp^1$ with geodesic metric,   $\TS(\Sp^1)$ is homeomorphic to the Hilbert cube. Next, we figure out the explicit form of those functions in $\TS(\Sp^1)$ which play the role of ``vertices" of the Hilbert cube. Finally, with this understanding, we prove that the homotopy type of $B_r(\Sp^1,\TS(\Sp^1))$ is that of $\Sp^1$ for $r\in \left(0,\frac{\pi}{3}\right]$ and that of a point for $r=\frac{\pi}{2}$.\vspace{\baselineskip}

In \S\ref{sec:mtrg}, we review the notion of mountain range function introduced by Katz in \cite{katz1991neighborhoods}, and with that we generate some examples of functions belonging to $\TS(\Sp^n)$ for arbitrary $n>2$.\vspace{\baselineskip}

In \S\ref{sec:TSSninfty}, we move our focus to $\Sp^n_\infty$, which is the $n$-sphere equipped with the $\ell_\infty$-metric instead of the geodesic metric. To understand $\TS(\Sp^n_\infty)$, we introduce the notions of $X$-surrounding points and $X$-minimal points associated to a metric space $X$.  In \cite{kilicc2016tight}, the authors prove that, for the circle $\Sp^1_\infty$,  $\TS(\Sp^1_\infty)$ is isometric to $\mathbb{D}^2_\infty$. This result suggests the question whether for all $n\geq 1$ it holds that $\TS(\Sp^n_\infty)$ is isometric to $\mathbb{D}^{n+1}_\infty$. We answer this question to the negative through an application of the notions of $X$-surrounding points and $X$-minimal points.  We prove that $\TS(\Sp^n_\infty)$ is not isometric to $\mathbb{D}^{n+1}_\infty$, $\mathbb{D}^{n+1}_\infty$ is not hyperconvex for $n\geq 2$ (See Theorem \ref{thm:ESn_infinity}), and also find an alternative description of $\TS(\Sp^2_\infty)$ (See Theorem \ref{thm:caE(S2)}).

\paragraph{Acknowledgements.} We acknowledge funding from these sources: NSF-DMS-1723003, NSF-CCF-1740761,
NSF-CCF-1839358, and BSF-2020124.

\section{Preliminaries}\label{sec:prelim}

Let us introduce some notations that will be used throughout this paper.

Between two metric spaces $X$ and $Y$, $X\cong Y$ denotes that $X$ is isometric to $Y$, and $X\preceq Y$ (resp. $X\prec Y$) denotes that $X$ can be isometrically embedded in $Y$ (resp. $X$ can be isometrically and properly embedded in $Y$).

Suppose that $X$ is a subspace of a metric space $(E,d_E)$. For every $r>0$, let $B_r(X,E):=\{z\in E:\,\exists x\in X\mbox{ with } d_E(z,x)<r\}$ denote the open $r$-thickening of $X$ in $E$. Respectively, the closed $r$-thickening of $X$ in $E$, denoted by $\overline{B}_r(X,E)$ is defined in the analogous way. In particular, if $X=\{x\}$ for some $x\in E$, it is just denoted by $B_r(x,E)$ (resp. $\overline{B}_r(x,E)$), the usual open (resp. closed) $r$-ball around $x$ in $E$.

As one more convention, whenever there is an isometric embedding $\iota:X\longhookrightarrow E$, we will use the notation $B_r(X,E)$ instead of $B_r(\iota(X),E)$.\vspace{\baselineskip}

Now we will review the concepts of \emph{injective} and \emph{hyperconvex} metric spaces. The main references for this subsection are \cite{dress1984trees,dress2012basic,lang2013injective}.

\begin{definition}[Injective metric space]\label{def:Injme}
A metric space $E$ is called \emph{injective} if for every $1$-Lipschitz map $f: X \to E$ and isometric embedding of $X$ into $\tilde{X}$, there exists a $1$-Lipschitz map $\tilde{f}: \tilde{X} \to E$ extending $f$:
$$\begin{tikzcd}
X \arrow[r, hook] \arrow[dr, "f" , rightarrow]
& \tilde{X} \arrow[d, "\tilde{f}"]\\
& E
\end{tikzcd}$$
\end{definition}

\begin{definition}[Hyperconvex space] \label{def:hyper}
A metric space $E$ is called  \emph{hyperconvex} if for every family $(x_i,r_i)_{i \in I}$ of $x_i$ in $E$ and $r_i \geq 0$ such that $d_E(x_i,x_j) \leq r_i + r_j$ for every $i,j$ in $I$, there exists a point $x\in E$ such that $d_E(x_i,x) \leq r_i$ for every $i$ in $I$.
\end{definition}

The following lemma is easy to deduce from the definition of hyperconvex space.

\begin{lemma}\label{lemma:intercloballhyper}
Any nonempty intersection of closed balls in a hyperconvex space is hyperconvex.
\end{lemma}

For a proof of the following proposition, see \cite{aronszajn1956extension} or \cite[Proposition 2.3]{lang2013injective}.

\begin{proposition}\label{prop:hypinj}
A metric space is injective if and only if it is hyperconvex.
\end{proposition}

Soon, we will provide an elegant proof of Proposition \ref{prop:hypinj} as one of the applications of tight span.

Moreover, every injective metric space is a contractible geodesic metric space, as one will see in Lemma \ref{lemma:geobicomb} and Corollary \ref{cor:injectivecont}.

\begin{definition}[Geodesic bicombing]
By a \emph{geodesic bicombing} $\gamma$ on a metric space $(X,d_X)$, we mean a continuous map $\gamma: X\times X\times [0,1]\rightarrow X$ such that, for every pair $(x,y)\in X\times X$, $\gamma(x,y,\cdot)$ is a geodesic from $x$ to $y$ with constant speed. In other words, $\gamma$ satisfies the following:
\begin{enumerate}
    \item $\gamma(x,y,0)=x$ and $\gamma(x,y,1)=y$.
    \item $d_X(\gamma(x,y,s),\gamma(x,y,t))=(t-s)\cdot d_X(x,y)$ for all $0\leq s\leq t\leq 1$.
\end{enumerate}
\end{definition}

\begin{lemma}[{\cite[Proposition 3.8]{lang2013injective}}]\label{lemma:geobicomb}
Every injective metric space $(E,d_E)$ admits a geodesic bicombing $\gamma$ such that, for all $x,y,x',y'\in E$ and $t\in [0,1]$, it satisfies the following conditions:
\begin{enumerate}
    \item \textbf{Conical:} $d_E(\gamma(x,y,t),\gamma(x',y',t))\leq (1-t)\,d_E(x,x')+t\,d_E(y,y')$. 
    \item \textbf{Reversible:} $\gamma(x,y,t)=\gamma(y,x,1-t)$.
    \item \textbf{Equivariant:} $L\circ \gamma(x,y,\cdot)=\gamma(L(x),L(y),\cdot)$ for every isometry $L$ of $E$.
\end{enumerate}
\end{lemma}

\begin{corollary}\label{cor:injectivecont}
Every injective metric space $E$ is contractible.
\end{corollary}
\begin{proof}
By Lemma \ref{lemma:geobicomb}, there is a geodesic bicombing $\gamma$ on $E$. Fix arbitrary point $x_0\in E$. Then, restricting $\gamma$ to $E \times \{x_0 \} \times [0,1]$ gives a deformation retraction of $E$ onto $x_0$. Hence, $E$ is contractible.
\end{proof}

\begin{example}
For every set $S$, the Banach space $L^\infty(S)$ consisting of all the bounded real-valued functions on $S$ with the metric induced by $\ell^\infty$-norm is injective. $\R^n_\infty$ is one of such spaces.
\end{example}

\begin{definition}\label{def:kuratowski}
For a compact metric space $(X,d_X)$, the map $\kappa_X:X \to L^\infty(X)$, $x \mapsto d_X(x,\cdot)$ is an isometric embedding, and it is called the \emph{Kuratowski embedding}. Hence, every compact metric space can be isometrically embedded into an injective metric space.
\end{definition}

There are other injective metric spaces associated to $(X,d_X)$.

\begin{definition}\label{def:inj}
Let $(X,d_X)$ be a metric space. Consider the following spaces associated to $X$:
\begin{align*}
&\Delta(X):=\{f\in L^\infty(X):f(x)+f(x')\geq d_X(x,x')\text{ for all }x,x'\in X\},\\
&\TS(X):=\{f\in\Delta(X):\text{if }g\in\Delta(X)\text{ and }g\leq f,\text{ then }g=f\,(\text{i.e., }f\text{ is minimal})\},\\
&\Delta_1(X):=\Delta(X)\cap\mathrm{Lip}_1(X), 
\end{align*}
where $\mathrm{Lip}_1(X)$ denotes the set of all 1-Lipschitz functions $f:X\rightarrow\mathbb{R}$. We endow these spaces with the metric induced by $\ell^\infty$-norm. Then, it is known that all the three spaces above are injective metric spaces (cf. \cite[Section 3]{lang2013injective}).

In particular, $\TS(X)$ is said to be the \textbf{tight span} of $X$ \cite{dress1984trees,isbell1964six} and it is an especially interesting space: $\TS(X)$ is the smallest injective metric space into which $X$ can be embedded, and it is unique up to isometry.
\end{definition}

\begin{remark}\label{rmk:propsofEX}
Here are some remarks. All of them are either straightforward, or can be found in \cite{lang2013injective}:
\begin{enumerate}
    \item By choosing $x=x'$ in the definition of $f\in \Delta(X)$, we see that any such $f$  is nonnegative \cite[\S 3]{lang2013injective}.
    
    \item $f\in\Delta_1(X)$ if and only if $\Vert f-d_X(x,\cdot) \Vert_\infty=f(x)$ for all $x\in X$ \cite[(3.2)]{lang2013injective}.
    
    \item For every $r>0$ and $E=\Delta_1(X)$ or $\TS(X)$, one has that $B_r(X,E)=\{f\in E:\exists\,x\in X\text{ such that }f(x)<r\}$.
    
    \item If $f\in \TS(X)$, then $f(x)=\sup_{x'\in X}(d_X(x,x')-f(x'))$ \cite[(3.1)]{lang2013injective}. In particular, this implies: $f(x)\in [0,\diam(X)]$ for all $x\in X$. With this and the fact that $X$ is isometrically embedded in $\TS(X)$, one can conclude $\diam(X)=\diam(\TS(X))$.
    
    \item If $f\in\TS(X)$ and $X$ is compact, then for all $x\in X$ there exists $x'\in X$ such that $d_X(x,x')=f(x)+f(x')$.

    \item There exists a ``projection" map  $p_X:\Delta(X)\rightarrow \TS(X)$ such that for all $f\in \Delta(X)$, $p_X(f)(x)\leq f(x)$ for all $x\in X$ \cite[Proposition 3.1]{lang2013injective}.
\end{enumerate}
\end{remark}

\begin{lemma}[{\cite[Proposition 3.5]{lang2013injective}}]\label{lm:iso span}
If $X$ is isometrically embedded in $Y$, then there exists an isometric embedding $h_{X,Y}:\TS(X)\hookrightarrow \TS(Y)$ such that for all $f\in \TS(X)$, $h_{X,Y}(f)|_X=f$.
\end{lemma}

At this step, we provide a proof of Proposition \ref{prop:hypinj} as promised. Our proof has a similar structure to \cite[Proposition, P182]{herrlich1992hyperconvex}, but by directly alluding to the tight span, our proof is significantly more concise.

\begin{proof}[Proof of Proposition \ref{prop:hypinj}]
Since a metric space $X$ is injective if and only if $X=\TS(X)$, it is enough to prove that $X$ is hyperconvex if and only if $X$ and $\TS(X)$ are isometric. 

Suppose $X$ is hyperconvex. Let $f\in \TS(X)$. Then, for all $x,x'\in X$, $f(x)+f(x')\geq d_X(x,x')$. By the hyperconvexity of $X$, $\cap_{x\in X} \overline{B}_{f(x)}(x,X)\neq\emptyset$. Let $x_0\in \cap_{x\in X} \overline{B}_{f(x)}(x,X).$ Then, $d_X(x_0,x)\leq f(x)$ for all $x\in X$. This implies that $f\geq d_X({x_0},\cdot)$. By the minimality of $\TS(X)$, we have $f=d_X({x_0},\cdot)$. This shows that $\TS(X)=X$.

Conversely, suppose $X=\TS(X)$. Consider any index set $I$ and let $X_I\coloneqq\{x_i\}_{i\in I}\subseteq X$ and $\{r_i\}_{i\in I}$ be such that $r_i+r_j\geq d_X(x_i,x_j)$. Then, the function $f_I:X_I\rightarrow \mathbb{R}$ taking $x_i$ to $r_i$ belongs to $\Delta(X_I)$. Consider the projection $p_I:\Delta(X_I)\rightarrow \TS(X_I)$ and the embedding $h_I:\TS(X_I)\hookrightarrow \TS(X)$ from Lemma \ref{lm:iso span}. Then,
$f\coloneqq h_I(p_I(f_I))\in \TS(X)=X$. Hence, there exists $x_0\in X$ such that $d_X({x_0},\cdot)=f$. Then,
$$d_X(x_0,x_i)= f(x_i)=p_I(f_I)(x_i)\leq f_I(x_i)=r_i$$
and thus $x_0\in \cap_{i\in I} \overline{B}_{r_i}(x_i,X)$. This implies that $X$ is hyperconvex.
\end{proof}

Finally, it is known that some topological/geometric properties of $X$ are reflected in $\TS(X)$.

\begin{theorem}[{\cite[2.11]{isbell1964six}}]
\label{thm:compactness_tight_span}
If $X$ is compact, then $\TS(X)$ is also compact. \end{theorem}

\begin{definition}[$\delta$-hyperbolic space]
A metric space $(X,d_X)$ is called \emph{$\delta$-hyperbolic}, for some constant $\delta\geq 0$, if
$$d_X(w,x)+d_X(y,z)\leq\max\{d_X(w,y)+d_X(x,z),d_X(x,y)+d_X(w,z)\}+\delta$$
for all quadruples of points $x,y,z,w\in X$.
\end{definition}

\begin{proposition}[{\cite[Proposition 1.3]{lang2013injective}}]\label{prop:hyperkura}
If $X$ is a $\delta$-hyperbolic geodesic metric space for some $\delta\geq 0$, then its tight span $\TS(X)$ is also $\delta$-hyperbolic.

Moreover,
$$B_r(X,\TS(X))=\TS(X)$$
for each $r>\delta$. 
\end{proposition}

\paragraph{Stability of the tight span.}It is known that the map $X\mapsto\TS(X)$ is stable with respect to the Gromov-Hausdorff distance.

\begin{theorem}[{\cite[Theorem 3.1]{lang2013metric}}]\label{thm:stability of tight span}
For any two metric spaces $A$ and $B$,
$$\dgh(\TS(A),\TS(B))\leq 2\,\dgh(A,B).$$
\end{theorem}

\subsection{Homotopy types of Vietoris-Rips complexes of tree-like metric spaces}

A $0$-hyperbolic metric space $X$ is called, alternatively, \emph{tree-like} metric space. i.e., a metric space $X$ is tree-like if it satisfies 
$$d_X(w,x)+d_X(y,z)\leq\max\{d_X(w,y)+d_X(x,z),d_X(x,y)+d_X(w,z)\}.$$
for every $x, y, z, w\in X$.

In this section, we show the homotopy types of Vietoris-Rips complexes of tree-like metric spaces can be easily determined using the knowledge of tight span. More precisely, we show that for a tree-like metric space $X$ and for any scale $r>0$, each connected component of the neighborhood $B_r(X, \TS(X))$ is contractible. Then we obtain the result for Vietoris--Rips complex $\vr_{2r}(X)$ by the identification of homotopy type between $\vr_{2r}(X)$ and $B_r(X, \TS(X))$ in \cite{lim2020vietoris}.

\begin{proposition}[{\cite[Proposition 2.3]{lim2020vietoris}}]\label{prop:vrbrhteqv}
Let $X$ be a subspace of an injective metric space $(E,d_E)$. Then, for every $r>0$, the Vietoris-Rips complex $\vr_{2r}(X)$ is homotopy equivalent to $B_r(X,E)$.
\end{proposition}

The tight span approach that we adopt in this section both generalizes the claim and offers an alternative technique to the one followed in the case of finite tree-like metric spaces \cite[Supp. Material]{chan2013topology}.

\begin{proposition}\label{prop:neighborhood_of_tree_metric_space}
Let $X$ be a tree-like metric space. Then for any $r>0$, the $r$-neighborhood of $X$ inside its tight span $\TS(X)$ is homotopy equivalent to a disjoint union of points with cardinality equal the number of connected components of $B_r(X, \TS(X))$.
\end{proposition}

\begin{corollary}\label{cor:httype_of_tree_metric_space}
Let $X$ be a tree-like metric space. Then for any $r>0$, $\vr_{2r}(X)$ is homotopy equivalent to a disjoint union of points with cardinality equal to the number of connected components of $B_r(X, \TS(X))$.
\end{corollary}
\begin{proof}
Apply Proposition \ref{prop:vrbrhteqv} and Proposition \ref{prop:neighborhood_of_tree_metric_space}.
\end{proof}

Our proof of the above proposition relies on the characterization of the tight span of a tree-like metric space as an \emph{$\R$-tree} \cite{dress1984trees}. We now recall the definition of an $\R$-tree. 

\begin{definition}\label{defn:R_tree}
Let $X$ be a metric space, $X$ is called an \emph{$\R$-tree} if the following two conditions are satisfied:
\begin{itemize}
    \item For every $x, x'\in X$ there exist a unique isometric embedding $\gamma:[0, d_X(x, x')]\rightarrow X$ such that $\gamma(0)= x$ and $\gamma(d_X(x, x')) = x'$.
    \item For every injective continuous map $\gamma:[0, 1]\hookrightarrow X$, for each $t\in [0, 1]$, one has $d_X(\gamma(0), \gamma(t)) + d_X(\gamma(t), \gamma(1)) = d_X(\gamma(0), \gamma(1))$.
\end{itemize}
\end{definition}

In \cite{dress1984trees}, Dress shows the tight span of a tree-like metric space is an $\R$-tree.

\begin{theorem}[{\cite[Theorem 8]{dress1984trees}}]\label{thm:tight_span_of_tree_metric}
Let $X$ be a tree-like metric space, then the tight span of $X$ is isometric to an $\R$-tree.
\end{theorem}

Now we are ready to prove Proposition \ref{prop:neighborhood_of_tree_metric_space}.

\begin{proof}[Proof of Proposition \ref{prop:neighborhood_of_tree_metric_space}]
    By Theorem \ref{thm:tight_span_of_tree_metric}, $\TS(X)$ is an $\R$-tree. Fix $r>0$. Let $C$ be a connected component of $B_r(X, \TS(X))$. As $\R$-trees are locally path connected, the connected component $C$ is path-connected. Therefore, $C$ is also an $\R$-tree and hence contractible.
\end{proof}

\subsection{Tight span of antipodal metric spaces}
In this section, we recall known results on the characterization of the tight span of antipodal metric spaces. These results are then used to  establish a simpler description of $\TS(\Sp^n)$ which in turn leads to a simpler function space $\mathbf{F}(\Sp^1)$ which is isometric to $\TS(\Sp^1)$ (see Definition \ref{def:redTSS1}) in \S\ref{sec:TSS1:description}. Also, this simpler description will be implicitly used in \S\ref{sec:mtrg}.

\begin{definition}
A metric space $X$ is called \emph{antipodal} if for every $x\in X$, there exists $\bar{x}\in X$ (called an \emph{antipodal point} of $x$) such that $d_X(x,\bar{x})=d_X(x,y)+d_X(y,\bar{x})$ holds for every $y\in X$. 
\end{definition} 

\begin{lemma}[{\cite[Lemma 7.7]{goodman2000tight}}]\label{lm:antipodal}
If $X$ is an antipodal compact metric space, then a function $f\in \TS(X)$ if and only if $f$ is 1-Lipschitz and $f(x)+f(\bar{x})=\diam(X),\forall x\in X$, where $\bar{x}$ is the antipodal point of $x$.
\end{lemma}

It is then easy to see that for any $f\in \TS(X)$, we have $0\leq f\leq \diam(X)$. This equation easily allows us to characterize the  \emph{center} and the \emph{radius} of $\TS(X)$.
\begin{definition}
For a bounded metric space $(X,d_X)$, we call $\mathrm{rad}(X)\coloneqq \inf_{x\in X}\sup_{y\in X} d_X(x,y)$ the \emph{radius} of $X$. If $x_0\in X$ satisfies that $\mathrm{rad}(X)=\sup_{y\in X}d_X(x_0,y)$, we say that $x_0$ is a \emph{center} of $X$.
\end{definition}

\begin{proposition}\label{prop:cts}
Let $X$ be a compact antipodal metric space. Then, $\mathrm{rad}(\TS(X))=\frac{\diam(X)}{2}$ and the constant function  $f_0\equiv\frac{\diam(X)}{2}$ is the \emph{unique} center of $\TS(X)$.
\end{proposition}
\begin{proof}
For every $f\in \TS(X)$ and every $x\in X$, we have that
\[\norm{f-d_{X}(x,\cdot)}_\infty=f(x)\text{ and }\norm{f-d_{X}(\bar{x},\cdot)}_\infty=f(\bar{x})=\diam(X)-f(x).\]
Then, 
\begin{align*}
   \sup_{g\in \TS(X)}\norm{f-g}_\infty&\geq \max\{\norm{f-d_{X}(x,\cdot)}_\infty,\norm{f-d_{X}(\bar{x},\cdot)}_\infty\}\\
   &=\max\{f(x),\diam(X)-f(x)\}\geq \frac{\diam(X)}{2}. 
\end{align*}
Moreover, if for a given $f\in\TS(X)$ there exists $x\in X$ such that $f(x)\neq \frac{\diam(X)}{2}$, we must have
\[\sup_{g\in \TS(X)}\norm{f-g}_\infty\geq \max\{f(x),\diam(X)-f(x)\}> \frac{\diam(X)}{2}.\]

Now, consider $f_0\equiv\frac{\diam(X)}{2}\in \TS(X)$. For each $f\in \TS(X)$ and each $x\in X$, since $0\leq f(x)\leq \diam(X)$, we have $\left|f(x)-\frac{\diam(X)}{2}\right|\leq \frac{\diam(X)}{2}$. This implies that $\norm{f-f_0}_\infty\leq \frac{\diam(X)}{2}$. Also, it is obvious that $\norm{d_{X}(x,\cdot)-f_0}_\infty= \frac{\diam(X)}{2}$. Hence,
$$\sup_{f\in \TS(X)} \norm{f_0-f}_\infty=\norm{f_0-d_{X}(x,\cdot)}_\infty=\frac{\diam(X)}{2}.$$
Therefore, $\mathrm{rad}(\TS(X))=\frac{\diam(X)}{2}$ and $f_0$ is the unique center of $\TS(X)$.
\end{proof}

We conclude this section with an interesting result about the convexity of the tight span of antipodal spaces.
\begin{proposition}\label{thm convex=antipodal}
Let $X$ be a compact metric space. Then, $X$ is antipodal if and only if $\TS(X)$ is \emph{convex} as a subset of $C(X)$, the space of all continuous functions $f:X\rightarrow\mathbb{R}$.
\end{proposition}

\begin{proof}
We first assume that $X$ is antipodal. Let $f,g\in \TS(X)$ and let $\lambda\in(0,1)$. Consider $h:= \lambda f+(1-\lambda)g\in C(X)$. Since both $f$ and $g$ are 1-Lipschitz, it is obvious that $h$ is 1-Lipschitz. Now, for every $x\in X$, we have that
\begin{align*}
    h(x)+h(\bar{x})&=\lambda\left( f(x)+f(\bar{x})\right)+(1-\lambda)\,\left( g(x)+g(\bar{x})\right)\\
    &=\lambda\,\diam(X)+(1-\lambda)\,\diam(X)\\
    &=\diam(X).
\end{align*}
Then, by Lemma \ref{lm:antipodal}, we have that $h\in \TS(X)$. Therefore, $\TS(X)$ is convex.

Now, we assume that $\TS(X)$ is a convex subspace of $C(X)$. Recall that $\kappa_X(x)=d_X(x,\cdot)\in \TS(X)$ for every $x\in X$. Since $X$ is compact, there exist $x_1,x_2\in X$ such that $d_X(x_1,x_2)=\diam(X)$. We will prove that $x_2=\bar{x}_1$. By convexity of $\TS(X)$, $f:=\frac{1}{2}(d_X(x_1,\cdot)+d_X(x_2,\cdot))\in \TS(X)$. Now, pick an arbitrary $y\in X$. Then, there exists $z\in X$ such that
\[f(y)=d_X(y,z)-f(z).\]
by item (4) of Remark \ref{rmk:propsofEX}. This implies that
\[2d_X(y,z)=\underbrace{d_X(x_1,y)+d_X(x_2,y)}_{\geq d_X(x_1,x_2)}+\underbrace{d_X(x_1,z)+d_X(x_2,z)}_{\geq d_X(x_1,x_2)}\geq d_X(x_1,x_2)+d_X(x_1,x_2)=2\diam(X).\]
Since $d_X(y,z)\leq \diam(X)$, we have that all equalities in the above formula holds. Therefore, $d_X(y,z)=\diam(X)=d_X(x_1,x_2)$ and $d_X(x_1,x_2)=d_X(x_1,y)+d_X(y,x_2).$

Since $y$ is arbitrary, we have that $x_2=\bar{x}_1$. Moreover, we have proved that for every $y\in X$, there exists $z\in X$ such that $d_X(y,z)=\diam(X)$. Then, we can apply the same argument as above to deduce that $z=\bar{y}$, which concludes the proof.
\end{proof}


\section{Characterization of $\TS(\Sp^1)$ and ancillary results}\label{sec:TSS1}

Given any positive integer $k$, let $C_{2k}$ denote the circular graph as shown in Fig. \ref{fig:c2k} with $2k$ vertices. We regard $C_{2k}$ as a metric space by equipping it with the shortest path distance. Then, \cite{goodman2000tight} established the following exact description of the tight span of $C_{2k}$.

\begin{figure}
    \centering
    \includegraphics[width=0.6\textwidth]{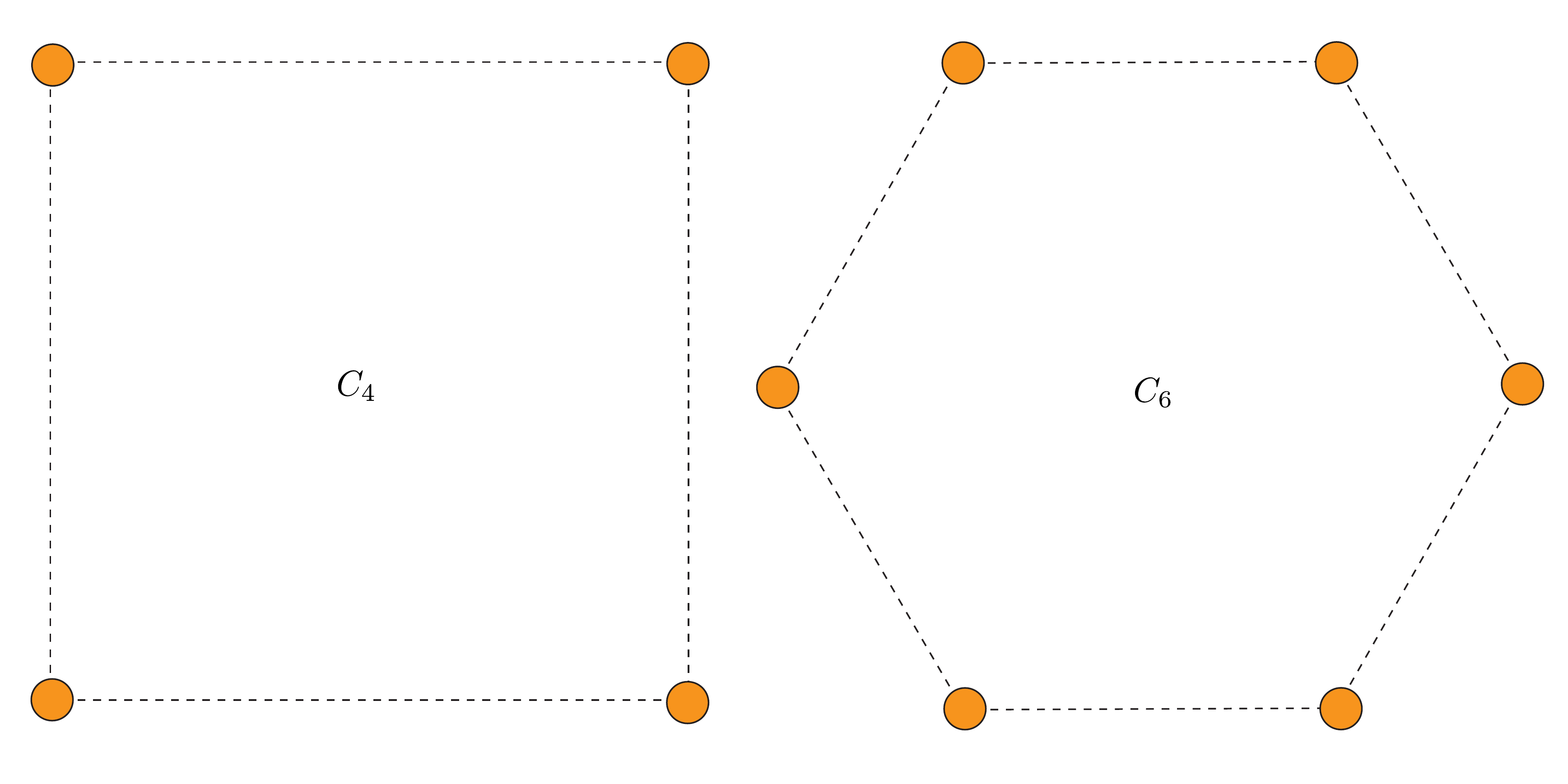}
    \caption{Examples of $C_{2k}$ when $k=2$ and $k=3$.}
    \label{fig:c2k}
\end{figure}

\begin{proposition}[{\cite[Section 9]{goodman2000tight}}]\label{prop:circular graph}
 $\TS(C_{{2k}})$ is a $k$-dimensional hypercube {living} in $\R^{2k}$. More precisely, $\TS(C_{{2k}})$ is the convex combination of the $2^k$ vertices described as follows: for every $\sigma=(\sigma_i)_{i=1}^k \in \{\pm 1\}^k$, consider $h_\sigma\in \TS(C_{2k})$ defined by
\begin{equation}\label{eq:h sigma}
    h_\sigma = \frac{k}{2}+\sum_{i=1}^k \sigma_i \cdot g_i 
\end{equation}
where $g_i$ is the function mapping vertices $i,i+1,\ldots,i+k-1$ to $\frac{1}{2}$, and the remaining vertices to $-\frac{1}{2}$. Then, the set of vertices of $\TS(C_{2k})$ is given by all the points $(h_\sigma(1),\ldots,h_\sigma(2k))\in \R^{2k}$ for  $\sigma \in \{\pm 1\}^k$.
\end{proposition}

Up to this point, the proposition above, and the fact that tight spans are Gromov-Hausdorff stable, suggests that the tight span of $\Sp^1$ should be, in a certain sense, an infinite dimensional hypercube. For any metric space $(X,d_X)$ and any $\lambda>0$, we denote by $\lambda X$ the metric space $(X,\lambda\cdot d_X)$. Then, it is obvious that,
\[\lim_{k\rightarrow\infty}\dgh\lbracket \Sp^1,\frac{\pi}{k}C_{2k}\rbracket=0.\]
In other words, $\Sp^1$ is the Gromov-Hausdorff limit of the sequence $\left\{\frac{\pi}{k}C_{2k}\right\}_{k=1}^\infty$. By Theorem \ref{thm:stability of tight span}, we conclude that $\TS(\Sp^1)$ is in fact the Gromov-Hausdorff limit of the sequence $\left\{\TS\lbracket\frac{\pi}{k}C_{2k}\rbracket\right\}_{k=1}^\infty$. For each $k=1,2,\ldots$, note that $\TS\lbracket\frac{\pi}{k}C_{2k}\rbracket=\frac{\pi}{k}\TS\lbracket C_{2k}\rbracket$ is a $k$-dimensional hypercube. Therefore, this suggests that $\TS(\Sp^1)$ should be an infinite-dimensional hypercube (in a  sense to be determined). 

In this section, we provide a precise description of $\TS(\Sp^1)$ and prove that $\TS(\Sp^1)$ is homeomorphic to the \emph{Hilbert cube}, a proper notion of an infinite-dimensional hypercube. Furthermore, we will clarify which functions in $\TS(\Sp^1)$ play the role of ``vertices" of the Hilbert cube and characterize the homotopy types of some open thickenings of $\Sp^1$ in $\TS(\Sp^1)$.

\subsection{A precise description of $\TS(\Sp^1)$}\label{sec:TSS1:description}

We will use $\Sp^n$ to denote the $n$-dimensional unit sphere with the geodesic (round) metric. Also, for $n=1$ case, we will view $\Sp^1$ as $\R\slash 2\pi$. As an abuse of notation, $\theta\in\R$ will also denote the equivalence class of $\theta$ in $\R\slash 2\pi=\Sp^1$.

By Lemma \ref{lm:antipodal}, we have that 
\begin{equation}\label{eq:tight span of S1}
    \TS(\Sp^1)=\{f\in\mathrm{Lip}_1(\Sp^1)|\,f(\theta)+f(\bar{\theta})=\pi,\,\forall \theta\in \Sp^1\}.
\end{equation}

In the sequel, we provide a  representation of $\TS(\Sp^1)$ simpler than Equation \eqref{eq:tight span of S1}, and then determine the homeomorphism type of $\TS(\Sp^1)$.

\begin{definition}[Reduced tight span of $\Sp^1$]\label{def:redTSS1}
Let 
$$\mathbf{F}(\Sp^1)\coloneqq\{f\in\mathrm{Lip}_1([0,\pi])|\,0\leq f\leq \pi, \, f(0)+f(\pi)=\pi\}.$$
Endow $\mathbf{F}(\Sp^1)$ with the sup norm. Then, we call $\mathbf{F}(\Sp^1)$ the \emph{reduced tight span of $\Sp^1$}.
\end{definition}

\begin{remark}
Actually, $\mathbf{F}(\Sp^1)=\{f\in\mathrm{Lip}_1([0,\pi])|\, f(0)+f(\pi)=\pi\}$. The condition that $0\leq f\leq \pi$ is  implied by $f(0)+f(\pi)=\pi$ and the 1-Lipschitz condition.
\end{remark}

\begin{remark}\label{rmk:compact convex}
It is easy to check that $\mathbf{F}(\Sp^1)$ is a compact convex subset of the space $C([0,\pi])$ consisting of all continuous maps from $[0,\pi]$ to $\R$.
\end{remark}

\begin{proposition}\label{prop:reduced tight span}
$\mathbf{F}(\Sp^1)$ is isometric to $\TS(\Sp^1)$.
\end{proposition}

To prove the theorem, we explicitly construct a map from $\mathbf{F}(\Sp^1)$ to $\TS(\Sp^1)$ as follows: for every $f\in\mathbf{F}(\Sp^1)$, we define $\tilde{f}:\Sp^1\rightarrow\R$ in the following way:
\begin{align*}
    \tilde{f}:\Sp^1&\longrightarrow\R\\
    \theta&\longmapsto\begin{cases}f(\theta)&\text{if }\theta\in [0,\pi]\\ \pi-f(\theta-\pi)&\text{if }\theta\in (\pi,2\pi]\end{cases}
\end{align*}

\begin{lemma}\label{lm:circle} 
If $f\in\mathbf{F}(\Sp^1)$, then $\tilde{f}$ belongs to $\TS(\Sp^1)$.
\end{lemma}
\begin{proof}
It suffices to show $\tilde{f}(\theta)+\tilde{f}(\bar{\theta})=\pi$ for all $\theta\in\Sp^1$ and $\tilde{f}$ is $1$-Lipschitz.

Fix arbitrary $\theta\in\Sp^1$. If $\theta\in [0,\pi]$,  $$\tilde{f}(\theta)+\tilde{f}(\bar{\theta})=f(\theta)+\tilde{f}(\theta+\pi)=f(\theta)+(\pi-f(\theta))=\pi.$$
If $\theta\in (\pi,2\pi]$,  $$\tilde{f}(\theta)+\tilde{f}(\bar{\theta})=(\pi-f(\theta-\pi))+\tilde{f}(\theta-\pi)=(\pi-f(\theta-\pi))+f(\theta-\pi)=\pi.$$ Hence, the first condition is satisfied.

Now, let's check that $\tilde{f}$ is $1$-Lipschitz. Fix arbitrary $\theta,\theta'\in\Sp^1$. If $\theta,\theta'\in [0,\pi]$, $$\vert\tilde{f}(\theta)-\tilde{f}(\theta')\vert=\vert f(\theta)-f(\theta')\vert\leq \vert \theta-\theta'\vert=d_{\Sp^1}(\theta,\theta').$$
If $\theta,\theta'\in [\pi,2\pi]$,
$$\vert\tilde{f}(\theta)-\tilde{f}(\theta')\vert=\vert (\pi-f(\theta-\pi))-(\pi-f(\theta'-\pi))\vert=\vert f(\theta-\pi)-f(\theta'-\pi)\vert\leq \vert \theta-\theta'\vert=d_{\Sp^1}(\theta,\theta').$$
Finally, if $\theta\in [0,\pi]$ and $\theta'\in [\pi,2\pi]$, we need a more subtle case-by-case analysis.

\begin{enumerate}
    \item If $\pi\leq\theta'\leq\theta+\pi$,
    \begin{align*}
        \vert\tilde{f}(\theta)-\tilde{f}(\theta')\vert&=\vert f(\theta)-(\pi-f(\theta'-\pi))\vert\\
        &\leq\vert f(\theta)+f(0)-\pi\vert+\vert f(\theta'-\pi)-f(0)\vert\\
        &=\vert f(\theta)-f(\pi)\vert+\vert f(\theta'-\pi)-f(0)\vert\\
        &\leq \vert\theta-\pi\vert+\vert\theta'-\pi\vert\\
        &=\vert\theta-\theta'\vert=d_{\Sp^1}(\theta,\theta').
    \end{align*}
    
    \item If $\theta+\pi\leq\theta'\leq 2\pi$,
    \begin{align*}
        \vert\tilde{f}(\theta)-\tilde{f}(\theta')\vert&=\vert f(\theta)-(\pi-f(\theta'-\pi))\vert\\
        &\leq\vert f(\theta'-\pi)+f(0)-\pi\vert+\vert f(\theta)-f(0)\vert\\
        &=\vert f(\theta'-\pi)-f(\pi)\vert+\vert f(\theta)-f(0)\vert\\
        &\leq \vert \theta'-2\pi \vert +\theta\\
        &=2\pi-\theta'+\theta=d_{\Sp^1}(\theta,\theta').
    \end{align*}
\end{enumerate}
This concludes the proof.
\end{proof}

\begin{proof}[Proof of Proposition \ref{prop:reduced tight span}]
By Lemma \ref{lm:circle}, $f\mapsto\tilde{f}$ is a well-defined map from $\mathbf{F}(\Sp^1)$ to $\TS(\Sp^1)$. To establish the claim, it is enough to show this map is surjective and preserves distances.

Let's prove the first claim. Fix arbitrary $g\in\TS(\Sp^1)$. Then, one can view $g$ as a map from $\R$ to $\R$ such that (1) $g(\theta+2\pi)=g(\theta)$ $\forall\theta\in\R$, (2) $g(\theta+\pi)=-g(\theta)$ $\forall\theta\in\R$, and (3) $g$ is $1$-Lipschitz. Let $f:=g\vert_{[0,\pi]}$. Then one can easily check $f\in\mathbf{F}(\Sp^1)$ and $\tilde{f}=g$.

Next, consider two arbitrary  functions $f,g\in\mathbf{F}(\Sp^1)$. Fix an arbitrary $\theta\in\Sp^1$. If $\theta\in [0,\pi]$, then $\vert \tilde{f}(\theta)-\tilde{g}(\theta)\vert=\vert f(\theta)-g(\theta)\vert$. If $\theta\in [\pi,2\pi]$, then $\vert \tilde{f}(\theta)-\tilde{g}(\theta)\vert=\vert (\pi-f(\theta-\pi))-(\pi-g(\theta-\pi))\vert=\vert f(\theta-\pi)-g(\theta-\pi)\vert$. Hence, $\Vert \tilde{f}-\tilde{g}\Vert_\infty=\Vert f-g\Vert_\infty$.

This concludes the proof.
\end{proof}

Already suggested by Proposition \ref{prop:circular graph}, via Proposition \ref{prop:reduced tight span} we next show that $\TS(\Sp^1)$ is an infinite-dimensional cube, namely, the {Hilbert cube}:
a\begin{definition}
Let $\ell^2$ denote the Hilbert space consisting of all infinite real sequences with finite square sums. The subspace
\[\mathfrak{H}\coloneqq\Pi_{n=1}^\infty\left[-\frac{1}{n},\frac{1}{n}\right]\subseteq\ell^2 \]
is called the \emph{Hilbert cube}.
\end{definition}

Recall that in a linear space $E$, the affine hull of any subset $S\subseteq E$ is defined as \[\mathrm{aff}(S)\coloneqq\left\{\sum_{i=1}^k{t_ix_i}:\,k>0,\, t_i\in\mathbb{R}\text{ and }\sum_{i=1}^kt_i=1\right\}.\]
For any convex $S$, the dimension of $S$ is defined to be the dimension of $\mathrm{aff}(S)$. 
The following result characterizes a large family of spaces homeomorphic to $\mathfrak{H}$.

\begin{lemma}[\cite{klee1955some}]\label{lm:infinite cube}
Any infinite-dimensional compact convex subset of a separable normed linear space is homeomorphic to the Hilbert cube.
\end{lemma}

\begin{lemma}\label{lm:infinite F}
$\mathbf{F}(\Sp^1)\subseteq C([0,\pi])$ is infinite dimensional.
\end{lemma}

\begin{proof}
It is easy to see that the affine hull $\mathrm{aff}(\mathbf{F}(\Sp^1))$ can be characterized as follows
$$\mathrm{aff}(\mathbf{F}(\Sp^1))=\{f\in\mathrm{Lip}([0,\pi])|\,f(0)+f(\pi)=\pi\},$$
where $\mathrm{Lip}([0,\pi])$ is the space of all Lipschitz functions (not necessarily 1-Lipschitz). $\mathrm{Lip}([0,\pi])$ is dense in $C([0,\pi])$ and thus infinite dimensional. Therefore $\mathrm{aff}(\mathbf{F}(\Sp^1))$ is a codimension-1 subspace of $\mathrm{Lip}([0,\pi])$ and thus of infinite dimension.
\end{proof}

Combining Remark \ref{rmk:compact convex}, Lemma \ref{lm:infinite cube}, and Lemma \ref{lm:infinite F}, we obtain the following topological characterization of $\mathbf{F}(\Sp^1)$.

\begin{theorem}\label{thm:hilbert cube}
$\mathbf{F}(\Sp^1)$ is homeomorphic to the Hilbert cube $\mathfrak{H}$.
\end{theorem}

As a direct consequence of this theorem, we have the following result.

\begin{theorem}\label{thm:E(Sn) is hilbert cube}
For every positive integer $n>0$, $\TS(\Sp^n)$ is homeomorphic to the Hilbert cube $\mathfrak{H}$.
\end{theorem}

\begin{proof}
By Proposition \ref{prop:reduced tight span} and Theorem \ref{thm:hilbert cube}, we have that $\TS(\Sp^1)$ is homeomorphic to $\mathfrak{H}$. 

Now, for any $n>1$, notice that $\TS(\Sp^n)$ is a compact convex subset of the separable normed linear space $C(\Sp^n)$ (cf. Theorem \ref{thm:compactness_tight_span} and Proposition \ref{thm convex=antipodal}). By Lemma \ref{lm:infinite cube}, we then only need to show that
$$\dim(\mathrm{aff}(\TS(\Sp^n)))=\infty.$$
Let $\Sp^1$ be any great circle in $\Sp^n$. By Lemma \ref{lm:iso span}, there exists an isometry $h_n:\TS(\Sp^1)\rightarrow \TS(\Sp^n)$ such that for every $f\in \TS(\Sp^1)$, $h_n(f)|_{\Sp^1}=f$. Then, $\mathrm{aff}(h_n(\TS(\Sp^1)))\subseteq\mathrm{aff}(\TS(\Sp^n))$. Consider the restriction map $\mathrm{res}:C(\Sp^n)\rightarrow C(\Sp^1)$ sending any $f\in C(\Sp^n)$ to $f|_{\Sp^1}$. The map $\mathrm{res}$ is obviously a linear map and it is easy to see that
\[\mathrm{res}(\mathrm{aff}(h(\TS(\Sp^n))))=\mathrm{aff}(\TS(\Sp^1)).\]
By Proposition \ref{prop:reduced tight span} and Lemma \ref{lm:infinite F}, we have that $\dim(\mathrm{aff}(\TS(\Sp^1)))=\infty$. This implies that 
$$\infty=\dim(\mathrm{aff}(h(\TS(\Sp^1))))\leq \dim(\mathrm{aff}(\TS(\Sp^n))), $$
which concludes the proof.
\end{proof}

The tight span of a metric space $X$ can be viewed as a ``hyperspace'' of $X$, i.e., a metric space containing $X$ as a subspace. The \emph{Hausdorff hyperspace} of $X$, which consists of nonempty closed subsets of $X$ and is endowed with the Hausdorff distance, is another type of hyperspace. It is known that the Hausdorff hyperspaces of many types of compact metric spaces are homeomorphic to $\mathfrak{H}$ \cite{schori1974hyperspaces,curtis1974,schori1975hyperspace}. This fact suggests us to also look into the tight span of more general compact metric spaces than spheres to determine whether an analogue of Theorem \ref{thm:hilbert cube} still holds. We leave this for future work.

\subsection{The vertex set of $\TS(\Sp^1)$}

From Proposition \ref{prop:circular graph} we know that $\TS(C_{2k})$ has the combinatorial structure of a hypercube, and we have an explicit description of its vertex set. In the previous section, we established that $\TS(\Sp^1)$ is an infinite-dimensional cube, i.e., the Hilbert cube. In this section, therefore, we provide an explicit description of the vertex set of $\TS(\Sp^1)$.
\begin{definition}\label{def:gtheta}
For every $\theta\in\R$, define $g_\theta:\Sp^1\rightarrow\R$ as follows
\begin{equation}\label{eq:gtheta}
    g_\theta(\varphi)\coloneqq\left\{
\begin{array}{rcl}
\frac{1}{2}       &      & \varphi\in[\theta,\theta+\pi)\\
-\frac{1}{2}     &      & {\mbox{otherwise}}
\end{array} \right.. 
\end{equation}
\end{definition}

As a generalization of Equation \eqref{eq:h sigma}, we identify the following family of functions defined on $\Sp^1$. Given any Lebesgue measurable $A\subseteq [0,\pi]$, define $h_A:\Sp^1\rightarrow \R$ as follows 
$$h_A(\varphi)\coloneqq\frac{\pi}{2}+\int_0^\pi g_\theta(\varphi)\big(1_A(\theta)-1_{[0,\pi]\backslash A}(\theta)\big)d\theta. $$

\begin{proposition}\label{prop:hA}
Given any Lebesgue measurable $A\subseteq [0,\pi]$, the function $h_A:\Sp^1\rightarrow \R$ defined above belongs to $\TS(\Sp^1).$
\end{proposition}
\begin{proof}
Note that $h_A$ is $1$-Lipschitz, since
$$\vert h_A(\alpha)-h_A(\beta)\vert\leq\int_0^\pi \vert g_\theta(\alpha)-g_\theta(\beta) \vert\,d\theta=d_{\Sp^1}(\alpha,\beta)$$
for every $\alpha,\beta\in\Sp^1$.

Now, for every $\theta\in[0,\pi]$ and $\varphi\in\Sp^1$, we have that $g_\theta(\varphi)+g_\theta(\bar{\varphi})=0$. This implies that $h_A(\varphi)+h_A(\bar{\varphi})=\pi$. Then, by Lemma \ref{lm:antipodal}, $h_A\in \TS(\Sp^1).$
\end{proof}

\begin{lemma}
Given two measurable subsets $A,B\subseteq[0,\pi]$ such that $\mu(A\Delta B)\neq 0$\footnote{Here $A \Delta B\coloneqq A\backslash B\cup B\backslash A$ denotes the symmetric difference between two sets.}, we have that $h_A\neq h_B$.
\end{lemma}

\begin{example}
For any $\alpha\in [0,\pi]$, let $A=[0,\alpha]$, then $h_A=d_{\Sp^1}(e^{i(\pi+\alpha)},\cdot)$ and $h_{[0,\pi]\backslash A}=d_{\Sp^1}(\alpha,\cdot)$.
\end{example}

Let $\mu$ denote the Lebesgue measure on $[0,\pi]$. Then, it is easy to observe the following explicit formula for $h_A$:

$$ h_A(\varphi)=\left\{
\begin{array}{lcl}
\frac{\pi}{2}-\frac{1}{2}\mu(A\cap[\varphi,\pi])+\frac{1}{2}\mu(A\cap[0,\varphi]) + \frac{1}{2}\mu([\varphi,\pi]\backslash A)-\frac{1}{2}\mu([0,\varphi]\backslash A),         & \varphi\in[0,\pi]\\
\frac{\pi}{2}+\frac{1}{2}\mu(A\cap[\bar{\varphi},\pi])-\frac{1}{2}\mu(A\cap[0,\bar{\varphi}]) -\frac{1}{2}\mu([\bar{\varphi},\pi]\backslash A)+\frac{1}{2}\mu([0,\bar{\varphi}]\backslash A),         & \varphi\in[\pi,2\pi]
\end{array} \right. 
$$
where $\bar{\varphi}=\varphi-\pi$ denotes the antipodal point of $\varphi$ in $\Sp^1$.

For every $\varphi\in(\pi,2\pi]$, $h_A(\varphi)=\pi-h_A(\varphi-\pi)$. So we focus on $\varphi\in[0,\pi]$. Then, the above formula can be further simplified as follows
$$h_A(\varphi)=h_A(0)+2\mu(A\cap[0,\varphi])-\varphi,$$
and $h_A(0)=\pi-\mu(A)$. 

\begin{definition}
In a convex set $S$, a point $x\in S$ is called an \emph{extreme point/vertex point}, if it is not the midpoint of any pair of distinct points $y,z\in S$.
\end{definition}

It is clear that a function $f\in \TS(\Sp^1)$ is an extreme point if and only if $f|_{[0,\pi]}$ is an extreme point in $\mathbf{F}(\Sp^1)$. Below, we analyze extreme points in $\mathbf{F}(\Sp^1)$. First of all, by Rademacher's theorem, $h_A|_{[0,\pi]}:[0,\pi]\rightarrow\mathbb{R}$ is differentiable a.e., and moreover by Lebesgue differentiation theorem, we have the following observation. 

\begin{lemma}\label{lm:hA derivative=1}
Given any Lebesgue measurable $A\subseteq[0,\pi]$, consider the derivative $(h_A|_{[0,\pi]})'$. Then, we have that
$(h_A|_{[0,\pi]})'(x)=1$ when $x\in A$ and $(h_A|_{[0,\pi]})'(x)=-1$ when $x\in [0,\pi]\backslash A$ a.e. 
\end{lemma}

It turns out that this phenomenon holds for every extreme point in $\mathbf{F}(\Sp^1)$:

\begin{proposition}
$h\in \mathbf{F}(\Sp^1)$ is an extreme point if and only if $|h'|\equiv 1$ on $[0,\pi]$ a.e.
\end{proposition}
\begin{proof}
Suppose first that $|h'|\equiv 1$ a.e.
Assume that there exist $f,g\in \mathbf{F}(\Sp^1)$ such that $\frac{f+g}{2}=h$. Let $A\coloneqq\{x\in [0,\pi]:\,h'(x)=1\}$. Then, $h'(x)=1$ for $x\in A$ and $h'(x)=-1$ for $x\in [0,\pi]\backslash A$ almost everywhere. Since both $f$ and $g$ are 1-Lipschitz, they are differentiable a.e. and $|f'|,|g'|\leq 1$. Since $\frac{f+g}{2}=h$, we have that $f'(x)=g'(x)=1$ for $x\in A$ a.e. and $f'(x)=g'(x)=-1$  for $x\in [0,\pi]\backslash A$ a.e. In other words, $f'=g'$ a.e. on $[0,\pi]$. Since $f(0)+f(\pi)=g(0)+g(\pi)=\pi$, by the fundamental theorem of calculus with respect to absolute continuous functions, we must have that $f=g=h$. Therefore, $h$ is an extreme point.

Conversely, assume that $h$ is an extreme point. Let $\mathrm{Lip}_1([0,\pi],0)$ denote the set of 1-Lipschitz functions $f:[0,\pi]\rightarrow\mathbb{R}$ such that $f(0)=0$. Then, $f\in \mathrm{Lip}_1([0,\pi],0)$ is an extreme point of $\mathrm{Lip}_1([0,\pi],0)$ if and only if $|f'|=1$, a.e. (cf. \cite[Proposition 6]{rolewicz1984optimal}). Now, it suffices to prove that $h_0\coloneqq h-h(0)$ is an extreme point in $\mathrm{Lip}_1([0,\pi],0)$ which implies that $|h'|=|(h-h(0))'|=1$ a.e. Assume that there exist $f_0,g_0\in \mathrm{Lip}_1([0,\pi],0)$ such that $h_0=\frac{f_0+g_0}{2}$. Define $f\coloneqq f_0+\frac{\pi-f_0(\pi)}{2}$ and $g\coloneqq g_0+\frac{\pi-g_0(\pi)}{2}$. It is easy to check that $f,g\in \mathbf{F}(\Sp^1)$ and that $h=\frac{f+g}{2}$. Since $h$ is an extreme point, we have that $f=g=h$ and thus $f_0=g_0=h_0$ which implies that $h_0$ is an extreme point as we required.
\end{proof}

\begin{corollary}
$h_A\in \TS(\Sp^1)$ is an extreme point for every measurable $A\subseteq [0,\pi]$.
\end{corollary}

\begin{lemma}
If $f\in \TS(\Sp^1)$ is an extreme point, then there is $A\subseteq [0,\pi]$ such that $f=h_A$.
\end{lemma}
\begin{proof}

Let $A\coloneqq\{x\in [0,\pi]:\,f'(x)=1\}$. Then, $f'(x)=1$ for $x\in A$ and $f'(x)=-1$ for $x\in [0,\pi]\backslash A$ a.e. Then, $(f|_{[0,\pi]})'=(h_A|_{[0,\pi]})'$ a.e. By the fundamental theorem of calculus for Lebesgue measure, we have that $f|_{[0,\pi]}=h_A|_{[0,\pi]}$ and thus $f=h_A$.
\end{proof}

Therefore, the set $\{h_A\}_{A\subseteq [0,\pi]}$ (where $A$ is required to be Lebesgue measurable) coincides with the set of all extreme points in $\TS(\Sp^1)$. Then, by the  Krein-Milman theorem, we deduce the following theorem which, in analogy with Proposition \ref{prop:circular graph}, permits interpreting $\{h_A\}_{A\subseteq [0,\pi]}$ as the ``set of vertices" of the infinite dimensional cube $\TS(\Sp^1)$ (cf. Theorem \ref{thm:hilbert cube}).  

\begin{theorem}
The closure of the convex hull of $\{h_A\}_{A\subseteq [0,\pi]}$ coincides with $\TS(\Sp^1)$.
\end{theorem}

\subsection{Homotopy types of open $r$-thickenings of $\Sp^1$ within $\TS(\Sp^1)$}\label{sec:TSS1:httypes}

Recall that, given any $r>0$, $B_r(\Sp^1,\TS(\Sp^1))$ denotes the open $r$-thickening of $\Sp^1$ in $\TS(\Sp^1)$. When $r>\frac{\pi}{2}$, we have that $B_r(\Sp^1,\TS(\Sp^1))=\TS(\Sp^1)$ which is contractible. When $r\in\left( 0,\frac{\pi}{2}\right)$, the homotopy type of $B_r(\Sp^1,\TS(\Sp^1))\subseteq \TS(\Sp^1)$ can be characterized as follows. By Proposition \ref{prop:vrbrhteqv} we have that $B_r(\Sp^1,\TS(\Sp^1))$ is homotopy equivalent to the Vietoris-Rips complex $\mathrm{VR}_{2r}(\Sp^1)$. Then, through the characterization of $\mathrm{VR}_{2r}(\Sp^1)$ in \cite{adamaszek2017vietoris}, we have that for any non-negative integer $n$, when $r\in\big(\frac{n\pi}{2n+1},\frac{(n+1)\pi}{2n+3}\big]$, $B_r(\Sp^1,\TS(\Sp^1))$ is homotopy equivalent to $\Sp^{2n+1}$. 

Note that the case when $r=\frac{\pi}{2}$ is not covered in the above argument. In this section, (1) we consider the case $r=\frac{\pi}{2}$, i.e., we prove that $B_{\frac{\pi}{2}}(\Sp^1,\TS(\Sp^1))$ is contractible, and (2) we provide a geometric proof, which does not rely on the Vietoris-Rips complex structure, of the fact that $B_r(\Sp^1,\TS(\Sp^1))$ is homotopy equivalent to $\Sp^1$ when $0<r\leq \frac{\pi}{3}$.

\paragraph{The case $r=\frac{\pi}{2}$.} We establish the following proposition.

\begin{proposition}\label{prop:contractible without 1 point}
$B_{\frac{\pi}{2}}(\Sp^1,\TS(\Sp^1))$ is contractible.
\end{proposition}

To prove this result, we first observe that the $r$-neighborhood of $\Sp^1$ in $\TS(\Sp^1)$ is in fact the complement of the $(\frac{\pi}{2}-r)$-neighborhood of the center $f_0\equiv\frac{\pi}{2}$ (cf. Proposition \ref{prop:cts}) in $\TS(\Sp^1)$:

\begin{lemma}\label{lm:neighborhood}
For every $r\in[0,\frac{\pi}{2}]$, ${B_r(\Sp^1,\TS(\Sp^1))}= \TS(\Sp^1)\backslash \overline{B}_{\frac{\pi}{2}-r}(f_0,\TS(\Sp^1))$, where $f_0\equiv\frac{\pi}{2}$.
\end{lemma}
\begin{proof}
If $f\in\overline{B}_{\frac{\pi}{2}-r}(f_0,\TS(\Sp^1))$, then $\left|f(\theta)-\frac{\pi}{2}\right|\leq\frac{\pi}{2}-r$ for every $\theta\in\Sp^1$, which implies that $r\leq f(\theta)\leq\pi-r$ for every $\theta\in\Sp^1$. Then, $\Vert f-d_{\Sp^1}(\theta,\cdot)\Vert=f(\theta)\geq r$ (cf. item 2 in Remark \ref{rmk:propsofEX}) for every $\theta\in\Sp^1$. Therefore, $f\notin{B_r(\Sp^1,\TS(\Sp^1))}$.

On the other hand, if $f\notin\overline{B}_{\frac{\pi}{2}-r}(f_0,\TS(\Sp^1))$, then $\exists\, \theta\in\Sp^1$ such that $\left|f(\theta)-\frac{\pi}{2}\right|>\frac{\pi}{2}-r$, which implies that $f(\theta)< r$ or $f(\theta)>\pi-r$. Then, 
\[\norm{f-d_{\Sp^1}(\theta,\cdot)}_\infty=f(\theta)<r\text{ or }\norm{f-d_{\Sp^1}(\bar{\theta},\cdot)}_\infty=f(\bar{\theta})=\pi-f(\theta)< r.\] Therefore, $f\in{B_r(\Sp^1,\TS(\Sp^1))}.$
\end{proof}

\begin{proof}[Proof of Proposition \ref{prop:contractible without 1 point}]
By Lemma \ref{lm:neighborhood}, $B_{\frac{\pi}{2}}(\Sp^1,\TS(\Sp^1))=\TS(\Sp^1)\backslash\{f_0\}$.
Therefore, $B_{\frac{\pi}{2}}(\Sp^1,\TS(\Sp^1))$ is the whole space $\TS(\Sp^1)$ minus one point. By Proposition \ref{prop:reduced tight span} and Theorem \ref{thm:hilbert cube}, $B_{\frac{\pi}{2}}(\Sp^1,\TS(\Sp^1))$ is homeomorphic to the Hilbert cube $\mathfrak{H}$ minus one point, which is of the same homotopy type as the Hilbert cube itself (cf. \cite[page 330]{chapman1972structure}) and is therefore contractible . 
\end{proof}

The proof of Lemma \ref{lm:neighborhood} above entirely relies on the facts that $\Sp^1$ is antipodal and $\diam(\Sp^1)=\pi$. Therefore, for every $n\in\mathbb{N}$, the proof (and thus the statement) of Lemma \ref{lm:neighborhood} can be easily generalized to the case of $\Sp^n$ for $n\geq 2$. In this way, by Theorem \ref{thm:E(Sn) is hilbert cube} and using a argument similar to the one used for proving Proposition \ref{prop:contractible without 1 point}, we conclude that:

\begin{proposition}
$B_{\frac{\pi}{2}}(\Sp^n,\TS(\Sp^n))$ is contractible for every $n\in\mathbb{N}$.
\end{proposition}

\paragraph{The case $r\in(0,\frac{\pi}{3}]$.} We now provide a direct proof\footnote{Here ``direct'' means a proof which does not invoke the result by Adams and Adamaszek described above.} of the fact that $B_r(\Sp^1,\TS(\Sp^1))$ is homotopy equivalent to $\Sp^1$ when $0<r\leq \frac{\pi}{3}$. This proof can be seen as being in the same spirit as but more direct than the one given by Katz in \cite[Section 3]{katz1991neighborhoods}. In fact, at the beginning of \cite[Section 3]{katz1991neighborhoods} the author  comments: 
\begin{quote}\emph{It is tempting to retract to the (suitably weighted) center of mass of the sublevel set $f\leq r$ which is contained in a semicircle. However,\ldots more work has to be done.}
\end{quote}
and then follows a different approach. Our proof departs from Katz's in that the strategy we follow is exactly that of considering such a weighted center of mass map (cf. equations (\ref{eq:bary}) and (\ref{eq:bary-proj}) below) and then directly checking that it gives us the required deformation retraction (cf. the proof of Theorem \ref{thm:def-retr}).\medskip

Recall that for every $f\in \TS(\Sp^1)$, we have that $f(\theta)=\Vert f-d_{\Sp^1}(\theta,\cdot) \Vert_\infty$ for any $\theta\in\Sp^2$ (cf. item 2 in Remark \ref{rmk:propsofEX}). Then, $f^{-1}([0,r))=\{\theta\in\Sp^1:\Vert f-d_{\Sp^1}(\theta,\cdot) \Vert_\infty<r\}$.

We will also use the following version of Jung's theorem:

\begin{theorem}[{Jung's Theorem \cite[Lemma 2]{katz1983filling}}]\label{thm:Jung}
Every subset of $\Sp^n$ with diameter less than or equal to $\arccos\left(-\frac{1}{n+1}\right)$ either coincides with the set of vertices of some inscribed regular $(n+1)$-simplex, or is contained in some ball of radius $\pi-\arccos\left(-\frac{1}{n+1}\right)$. 
\end{theorem}

\begin{lemma}\label{lm:semi-circle}
If $r\in(0,\frac{\pi}{3}]$ and $f\in \TS(\Sp^1)$, $f^{-1}([0,r))$ is contained in an arc with length no longer than $2r$.
\end{lemma}
\begin{proof}
For every $\theta,\theta'\in f^{-1}([0,r))$, one has that $\norm{f-d_{\Sp^1}(\theta,\cdot)}_\infty,\norm{f-d_{\Sp^1}(\theta',\cdot)}_\infty<r$. Then,
\[d_{\Sp^1}(\theta,\theta')=\norm{d_{\Sp^1}(\theta,\cdot)-d_{\Sp^1}(\theta',\cdot)}_\infty\leq \norm{f-d_{\Sp^1}(x,\cdot)}_\infty+\norm{f-d_{\Sp^1}(x',\cdot)}_\infty<2r.\]
This implies that $\diam(f^{-1}([0,r)))\leq 2r\leq \frac{2\pi}{3}.$ By Jung's theorem (cf. Theorem \ref{thm:Jung}), $f^{-1}([0,r))$ is contained in a semicircle and thus inside an arc with length not exceeding  $2r$.
\end{proof}

Let $s:\R\rightarrow\R\slash 2\pi=\Sp^1$ be the canonical quotient map. Given $r\in (0,\frac{\pi}{3}]$, let $f$ be in $ B_r(\Sp^1,\TS(\Sp^1))$, then $f^{-1}([0,r))$ is contained in a semicircle (cf. Lemma \ref{lm:semi-circle}). Then, there exists $\theta_f\in[0,2\pi)$ such that $f^{-1}([0,r))\subseteq s([\theta_f,\theta_f+\pi])$. Pick any such $\theta_f$ for every $f\in B_r(\Sp^1,\TS(\Sp^1))$. \medskip

Consider the following \emph{weighted center of mass:} 

\begin{equation}\label{eq:bary}n_f^r \coloneqq \frac{\int_{\theta_f}^{\theta_f+\pi}x(r-f\circ s(x))1_{f^{-1}([0,r))}(s(x))\,dx}{\int_{\theta_f}^{\theta_f+\pi}(r-f\circ s(x))1_{f^{-1}([0,r))}(s(x))\,dx}\in\R\end{equation}
and let

\begin{equation}\label{eq:bary-proj}m_f^r\coloneqq s(n_f^r)\in\Sp^1. \end{equation}

See Figure \ref{fig:barycenter} for an illustration of $n_f^r$ when $\theta_f=0$. Note that since $f\in B_r(\Sp^1,\TS(\Sp^1))$, the denominator in Equation \eqref{eq:bary} is positive and thus $n_f^r$ is well-defined. Recall that $\theta_f\in[0,2\pi)$ is arbitrarily chosen such that $f^{-1}([0,r))\subseteq s([\theta_f,\theta_f+\pi])$. It is also not hard to see that $n_f^r$ does not depend on the choice of $\theta_f$.

Moreover,  we have \begin{equation}\label{eq:restr}m_{d_{\Sp^1}(\theta,\cdot)}^r=\theta\,\,\mbox{for every $\theta\in\Sp^1$}.\end{equation}

\begin{figure}
    \centering
    \includegraphics[width=0.5\textwidth]{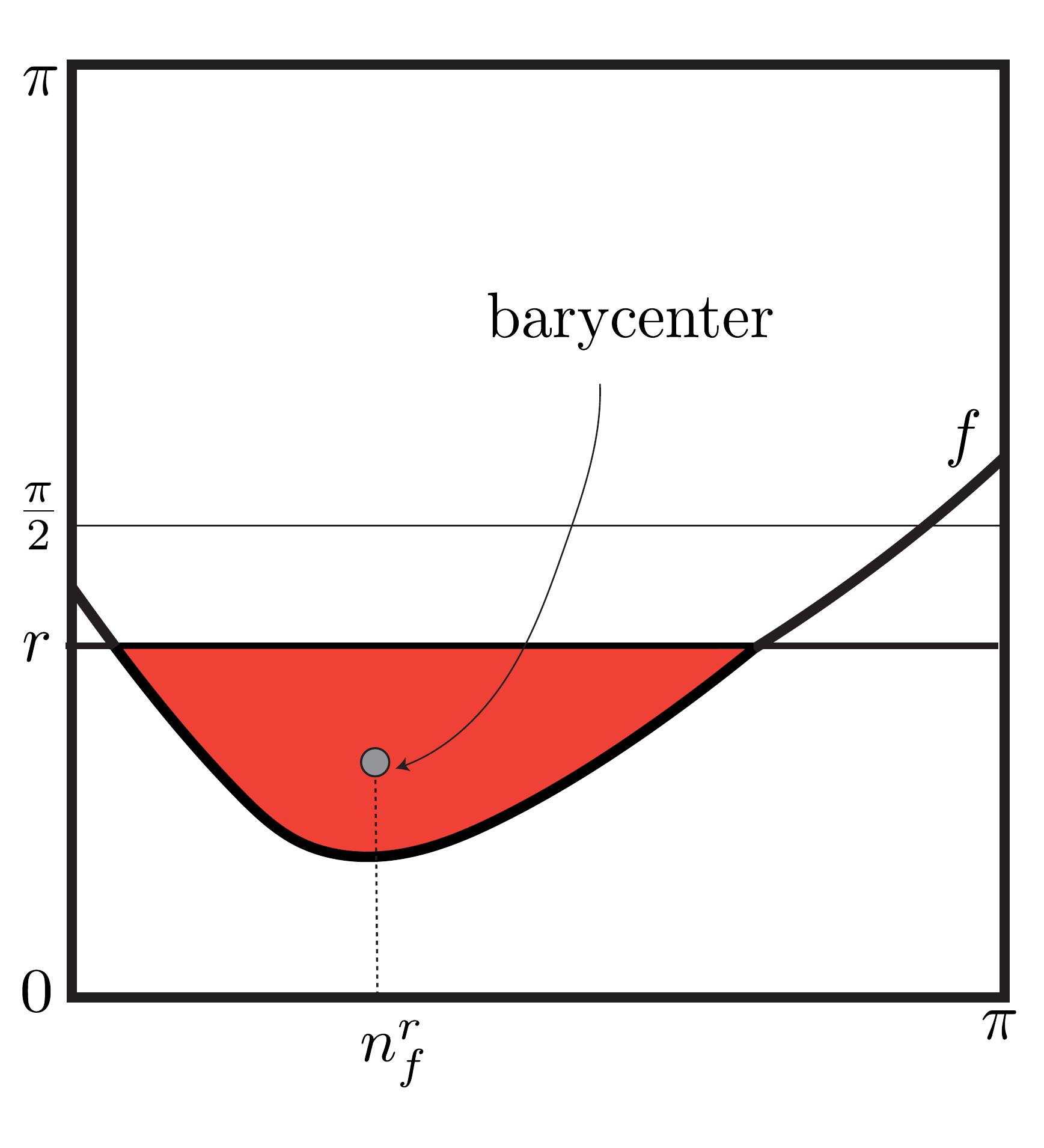}
    \caption{The construction of $n_f^r$. The real value $n_f^r$ is the projection of the barycenter of the red region onto the the x-axis. See equations (\ref{eq:bary}) and (\ref{eq:bary-proj}) below.}
    \label{fig:barycenter}
\end{figure}

\begin{lemma}
If $r\in(0,\frac{\pi}{3}]$, the map $m^r_\bullet:B_r(\Sp^1,\TS(\Sp^1))\rightarrow\Sp^1$  taking $f$ to $m_f^r$ is a retraction.
\end{lemma}
\begin{proof}
We only need to prove the continuity of $m^r_\bullet$. Fix an  arbitrary $f\in B_r(\Sp^1,\TS(\Sp^1))$ and $\eps>0$. Next, choose an arbitrary $g\in B_r(\Sp^1,\TS(\Sp^1))$ such that $\norm{f-g}_\infty<\eps$. Let $B_f\coloneqq f^{-1}([0,r))$, $B_g\coloneqq g^{-1}([0,r))$ and $B\coloneqq B_f\cap B_g.$ Consider any $x\in s^{-1}(B_f\backslash B)$. Then, $g\circ s(x)\geq r$ and $g\circ s(x)-f\circ s(x)<\eps$. Thus, $r-f\circ s(x)<\eps$. Similarly, for every $x\in s^{-1}(B_g\backslash B)$, $r-g\circ s(x)<\eps$.

Now, we show that $\ds(m_f^r,m_g^r)\leq C_f\eps$ where $C_f>0$ is a constant depending on $f$. This will establish the continuity of $m_\bullet^r$.  By the requirement that $\theta_f,\theta_g\in [0,2\pi)$, there are only two possible cases:

\begin{enumerate}
    \item $s^{-1}(B)\cap[\theta_f,\theta_f+\pi]=s^{-1}(B)\cap[\theta_g,\theta_g+\pi].$
    
    \item $s^{-1}(B)\cap[\theta_f,\theta_f+\pi]=s^{-1}(B)\cap[\theta_g,\theta_g+\pi]\pm 2\pi.$
\end{enumerate}

For the first case, we have that

\begin{align*}
    &\left|\int_{\theta_f}^{\theta_f+\pi}(r-f\circ s(x))1_{B_f}(s(x))\,dx-\int_{\theta_g}^{\theta_g+\pi}(r-g\circ s(x))1_{B_g}(s(x))\,dx\right|  \\
    =&\Bigg|\left(\int_{[\theta_f,\theta_f+\pi]\cap s^{-1}(B)}+\int_{[\theta_f,\theta_f+\pi]\cap s^{-1}(B_f\backslash B)}\right)(r-f\circ s(x))\,dx\\
    -&\left(\int_{[\theta_g,\theta_g+\pi]\cap s^{-1}(B)}+\int_{[\theta_g,\theta_g+\pi]\cap s^{-1}(B_g\backslash B)}\right)(r-g\circ s(x))\,dx\Bigg|\\
    \leq&\int_{[\theta_f,\theta_f+\pi]\cap s^{-1}(B)}|g\circ s(x)-f\circ s(x)|\,dx+\int_{[\theta_f,\theta_f+\pi]\cap s^{-1}(B_f\backslash B)}|r-f\circ s(x)|\,dx\\
    +&\int_{[\theta_g,\theta_g+\pi]\cap s^{-1}(B_g\backslash B)}|r-g\circ s(x)|\,dx\\
    \leq &3\pi\eps,
\end{align*}
Similarly,
\[\left|\int_{\theta_f}^{\theta_f+\pi}x(r-f\circ s(x))1_{B_f}(s(x))\,dx-\int_{\theta_g}^{\theta_g+\pi}x(r-g\circ s(x))1_{B_g}(s(x))\,dx\right|\leq 9\pi^2\eps. \]

For notational simplicity, we let $$S_f\coloneqq{\int_{\theta_f}^{\theta_f+\pi}(r-f\circ s(x))1_{f^{-1}([0,r))}(s(x))\,dx}$$ and  $$E_f\coloneqq{\int_{\theta_f}^{\theta_f+\pi}x(r-f\circ s(x))1_{f^{-1}([0,r))}(s(x))\,dx}.$$ We also define $S_g$ and $E_g$ in a similar manner. Then, $|S_f-S_g|\leq 3\eps$ and $|E_f-E_g|\leq 9\pi\eps^2$. Since $S_f>0$, we can choose $\eps$ small enough so that $S_g\geq S_f-3\eps>\frac{S_f}{2}$. With this choice of $\varepsilon$, we  then have
\begin{align*}
|n_f^r-n_g^r|&=
    \left|\frac{E_f}{S_f}-\frac{E_g}{S_g}\right|= \left|\frac{E_fS_g-E_gS_f}{S_fS_g}\right|\\
    &\leq 2\left|\frac{E_f(S_g-S_f)}{S_f^2}\right|+2\left|\frac{(E_f-E_g)S_f}{S_f^2}\right|\\
    &\leq \frac{2(3\pi E_f+9\pi^2S_f)}{S_f^2}\cdot\eps
\end{align*}

Hence, by the continuity of the quotient map $s$, we have that $\ds(m_f^r,m_g^r)\leq C_f\,\eps,$ where the constant $C_f\coloneqq \frac{2(3\pi E_f+9\pi^2S_f)}{S_f^2}>0$ depends only on $f$.

Now for the second case, without loss of generality one can assume $s^{-1}(B)\cap[\theta_f,\theta_f+\pi]=s^{-1}(B)\cap[\theta_g,\theta_g+\pi]- 2\pi.$ By the periodicity of $g\circ s$ and $s$, one has 
$$\frac{\int_{\theta_g}^{\theta_g+\pi}x(r-g\circ s(x))1_{B_g}(s(x))\,dx}{\int_{\theta_g}^{\theta_g+\pi}(r-g\circ s(x))1_{B_g}(s(x))\,dx} = \underbrace{\frac{\int_{\theta_g-2\pi}^{\theta_g+\pi-2\pi}x(r-g\circ s(x))1_{B_g}(s(x))\,dx}{\int_{\theta_g-2\pi}^{\theta_g+\pi-2\pi}(r-g\circ s(x))1_{B_g}(s(x))\,dx}}_{=n'_g}+2\pi.$$
Then $m_g^r=s(n_g^r)=s(n'_g).$ Take $\theta'_g=\theta_g-2\pi$. Then $s^{-1}(B)\cap[\theta_f,\theta_f+\pi]=s^{-1}(B)\cap[\theta'_g,\theta'_g+\pi]$ and thus the argument for the first case applies here to conclude that $\ds(m_f^r,m_g^r)\leq C_f'\eps,$ where $C_f'>0$ is some constant depending only on $f$.

\medskip
In conclusion, $m^r_\bullet:B_r(\Sp^1,\TS(\Sp^1))\rightarrow\Sp^1$ is continuous and, in view of (\ref{eq:restr}), it is therefore a retraction.
\end{proof}

\begin{lemma}\label{lm:middle point}
For every $r\in(0,\frac{\pi}{3}]$ and $f\in B_r(\Sp^1,\TS(\Sp^1))$, there exists $\theta\in\Sp^1$ such that $\ds(\theta,m^r_f)< r$ and $\norm{f-d_{\Sp^1}(\theta,\cdot)}_\infty<r$.
\end{lemma}

\begin{proof}
In fact, this follows directly from Lemma \ref{lm:semi-circle}. Since $f^{-1}([0,r))$ is contained in an arc with length no longer than $2r$, $m^r_f$ also lies in this arc. More precisely, $m^r_f$ belongs to the geodesic convex hull of $f^{-1}([0,r))$. Then, there exists $\theta\in f^{-1}([0,r))$ such that $\ds(\theta,m^r_f)<r$. Since $\theta\in f^{-1}([0,r))$, one has that $\norm{f-d_{\Sp^1}(\theta,\cdot)}_\infty=f(\theta)<r.$
\end{proof}

\begin{theorem}\label{thm:def-retr}
For every $r\in(0,\frac{\pi}{3}]$, $B_r(\Sp^1,\TS(\Sp^1))$ is homotopy equivalent to $\Sp^1$.
\end{theorem}
\begin{proof}
Recall that $\kappa_{\Sp^1}:\Sp^1\rightarrow B_r(\Sp^1,\TS(\Sp^1))$ denotes the Kuratowski embedding.
We build a homotopy $$H:B_r(\Sp^1,\TS(\Sp^1))\times [0,1]\rightarrow B_r(\Sp^1,\TS(\Sp^1))$$ between $\kappa_{\Sp^1}\circ m^r:B_r(\Sp^1,\TS(\Sp^1))\rightarrow B_r(\Sp^1,\TS(\Sp^1))$ and $\mathrm{id}:B_r(\Sp^1,\TS(\Sp^1))\rightarrow B_r(\Sp^1,\TS(\Sp^1))$ as follows. For every $(f,t)\in B_r(\Sp^1,\TS(\Sp^1))\times [0,1]$, let $H(f,t)\coloneqq tf+(1-t)d_{\Sp^1}(m^r_f,\cdot).$ We need to show $H$ is well-defined, i.e., the image of $H$ lies in $B_r(\Sp^1,\TS(\Sp^1))$. In fact, by Lemma \ref{lm:middle point}, there is $x\in\Sp^1$ such that $\norm{d_{\Sp^1}(x,\cdot)-f}_\infty<r$ and $\norm{d_{\Sp^1}(x,\cdot)-d_{\Sp^1}(m^r_f,\cdot)}_\infty<r.$ Hence
$$\norm{tf+(1-t)d_{\Sp^1}(m^r_f,\cdot)-d_{\Sp^1}(x,\cdot)}_\infty\leq t\norm{f-d_x}_\infty+(1-t)\norm{d_{\Sp^1}(x,\cdot)-d_{\Sp^1}(m^r_f,\cdot)}_\infty<r. $$
Thus $H(f,t)\in B_r(\Sp^1,\TS(\Sp^1))$. It is then easy to check that $H$ is continuous and thus $\Sp^1$ is a deformation retraction of $B_r(\Sp^1,\TS(\Sp^1))$.
\end{proof}

\section{Mountain range functions and the tight span}\label{sec:mtrg}
In this section, we first review the notion of \emph{mountain range function} that was introduced by Katz in \cite{katz1991neighborhoods}. We then study the problem of determining sufficient conditions for a mountain range function on a sphere to be an element in the tight span of sphere. 

\begin{definition}
    Let $X$ be a metric space and $P=\{p_1,\cdots, p_n\}\subseteq X$ be a finite subset along with a set of real numbers $V=\{v_1, \cdots, v_n\}$. We define the \emph{mountain range function} $\mr_X(P, V):X\rightarrow \R$ as
    \begin{equation*}
        \mr_X(P, V)(x) := \min_{i = 1,\cdots ,n} (d_X(x, p_i) + v_i).
    \end{equation*}
\end{definition}

\begin{remark}
    For a mountain range function $\mr_X(P, V)$, if all $v_i$ are equal to the same number $v$, obtain that $\mr_X(P, V)(x)$ equals $d_X(x, P) + v$ where $d_X(x, P)$ denotes the distance function to the set $P$.
\end{remark}

Now, we give conditions under which a mountain range function on $X$ is an element in $\Delta_1(X)$  (cf. Definition \ref{def:inj}).

\begin{proposition}\label{prop:mr_in_D1}
 Let $X$ be a compact metric space and let $P$ be a finite subset equipped with the induced metric. For every function $f\in \Delta(P)$, we associate it with a mountain range function $\mr_X(P,f)$ in the following way,
    \[\mr_X(P, f)(x) := \min_{p_i\in P}(d_X(x, p_i) + f(p_i))\]
    Then $\mr_X(P, f) \in \Delta_1(X)$. Moreover, if $f\in \Delta_1(P)$, then $\mr_X(P, f)$ is an extension of $f$ onto $X$.
\end{proposition}

\begin{proof}
    Let $f$ be a function in $\Delta(P)$, then for every $p_i, p_j\in P$, we have $f(p_i) + f(p_j)\geq d_X(p_i, p_j)$. We first show that $\mr_X(P, f)\in \Delta(X)$. For any $x, x'$ in $X$, there are some $i, j$ such that $\mr_X(P, f)(x) = d_X(x, p_{i}) + f(p_{i})$ and $\mr_X(P, f)(x') = d_X(x', p_j) + f(p_j)$. Therefore,
    \begin{align*}
        \mr_X(P, f)(x) + \mr_X(P, f)(x') & = d_X(x, p_i) + f(p_i) + d_X(x', p_j) + f(p_j)\\
        & \geq d_X(x, p_i) + d_X(p_i, p_j) + d_X(p_{i'}, x')\\
        & \geq d_X(x, x').
    \end{align*}
    As $X$ is compact, $\mr_X(P, f)$ is a bounded function and hence, $\mr_X(P, f)\in \Delta(X)$.

    Next, we need to show $\mr_X(P, f)$ is a $1$-Lipschitz function. For any $x, x'$ in $X$, let $p_j\in P$ be such that $\mr_X(P, f)(x') = d_X(x, p_j) + f(p_j)$, as before. Then, we have
    \begin{align*}
        \mr_X(P, f)(x) - \mr_X(P, f)(x') & \leq (d_X(x, p_j) + f(p_j)) - (d_X(x', p_j) - f(p_j)) \\
        & = d_X(x, p_j) - d_X(x', p_j) \\
        & \leq d_X(x, x').
    \end{align*}
    Similarly, we have $\mr_X(P, f)(x') - \mr_X(P, f)(x)\leq d_X(x, x')$. This shows that $\mr_X(P, f)$ is a $1$-Lipschitz function and $\mr_X(P, f)$ belongs $\Delta_1(X)$.

    If $f$ is inside $\Delta_1(X)$, then $f$ is furthermore $1$-Lipschitz. For fixed $p_i\in P$ and any other $p_j\in P$, we have
    \begin{align*}
        d_X(p_i, p_i) + f(p_i) &\leq f(p_j) + d_X(p_i, p_j).
    \end{align*}
    Therefore, $\mr_X(P, f)(p_i) = f(p_i)$ for every $p_i\in P$, that is, $\mr_X(P, f)$ is an extension of $f$.
\end{proof}

\begin{remark}
    The above construction gives the following commutative diagram:
    \begin{center}
        \begin{tikzcd}
            {P} \arrow[r, hook] \arrow[d, hook, "\kappa_P"'] & X \arrow[d, hook, "\kappa_X"] \\
            \Delta_1(P) \arrow[r, hook, "\mr"] & \Delta_1(X)
        \end{tikzcd}
    \end{center}
    where the vertical maps are Kuratowski embeddings.
\end{remark}

Now we focus on the case where $X$ is a sphere and $f$ is a constant function. It turns out that for $\mr_{\Sp^n}(P, f)$ to be an element in the tight span $\TS(\Sp^n)$, $P$ must be a special configuration that we will now introduce. For any two points $x, x'\in P$, we say that $x$ and $x'$ are \emph{comaximal} if $d_X(x, x') = \diam(P)$. We use the notation $\comax_P(x)$ to denote the set of points $x'$ in $P$ that is comaximal with $x$.
 
\begin{definition}[\cite{katz1989diameter}]
    Let $P$ be a subset of $\Sp^n$ containing no antipodal pairs. We say that a point $p\in P$ is \emph{held} by $P$ if for every tangent vector $v\in T_p\Sp^n$, there exists a point $p'\in \comax_P(p)$ such that the inner product $\langle v, \exp^{-1}_p(p')\rangle_{T_p\Sp^n} \geq 0$,  where $\exp^{-1}_p(p')$ is the unit tangent vector at $p$ that is tangent to the unique shortest geodesic connecting $p$ to $p'$.

    We say that $P$ is \emph{pointwise extremal} if every point $p\in P$ is held by $P$.
\end{definition}

\begin{remark}
    In $\Sp^1$, a finite subset is pointwise extremal if and only if it is the vertex set of an inscribed regular odd $n$-gon, see \cite[Lemma~4.3]{katz1991neighborhoods}.
\end{remark}

\begin{proposition}\label{proposition:held_point_as_local_max}
    If $P$ is a finite subset of $\Sp^n$ with no antipodal pairs. Then $P$ is pointwise extremal if and only if for every $p_i\in P$, the function $d_{\Sp^n}(\cdot, P)$ attains a local maximum at $\bar{p_i}$, the antipodal of $p_i$.
\end{proposition}
\begin{proof}
     Let $p\in P$ be a fixed point in $P$. If $p$ is held by $P$, then for every $v\in T_{p}\Sp^n$, there exists a point $p_{i}\in \comax_P(p)$ such that the inner product $\langle-v,  \exp_{p}^{-1}(p_{i})\rangle \geq 0$, that is $\langle v , \exp_{p}^{-1}(p_{i})\rangle \leq 0$. This is equivalent to the function
     \[\Phi_{p}(x)= \max_{p_{i}\in \comax_P(p)}d_{\Sp^n}(p_{i}, x)\]
     obtaining a local minimum at $p$. Note that, as $P$ is a finite set, for every $x$ in a small neighborhood around $\bar{p}$, we have
     \[d_{\Sp^n}(x, P) = \min_{p_i\in  \comax_P(p)} d_{\Sp^n}(x, p_i) = \min_{p_i\in  \comax_P(p)}(\pi - d_{\Sp^n}(\bar{x}, p_i)) = \pi- \Phi_{p}(\bar{x}).\]
    Therefore, $d_{\Sp^n}(\cdot, P)$ attaining a local maximum at $\bar{p}$ is equivalent to $\Phi_{p}(\cdot)$ attaining a local minimum at $p$ which is equivalent to $p$ being held by $P$.
\end{proof}

\begin{proposition}
    Let $P$ be a finite subset of $\Sp^n$ without antipodal pairs. If the function
    \[\mr_{\Sp^n}(P, a)(x) = d_{\Sp^n}(x, P) + a\]
    is inside $\TS(\Sp^n)$, then $a$ equals $\frac{1}{2}\diam(P)$ and $P$ is a pointwise extremal set.
\end{proposition}
\begin{proof}
    Fix a $p$ in $P$ such that there is some $p'\in P$ with $d_{\Sp^n}(p, p') = \diam(P)$. Then we have $d_{\Sp^n}(\bar{p}, P) = \pi - \diam(P)$.
    As $\mr_{\Sp^n}(P, a)(x)\in \TS(\Sp^n)$, we have
    \[\pi = \mr_{\Sp^n}(P, a)(p) + \mr_{\Sp^n}(P, a)(\bar{p}) = a + d_{\Sp^n}(\bar{p}, P) + a\]
    This shows $d_{\Sp^n}(\bar{p}, P) = \pi - 2\, a$, that is, $a = \frac{1}{2}\,(\pi - d_{\Sp^n}(\bar{p}, P)) = \frac{1}{2}\diam(P)$.
    For every $p_i\in P$, note that, the function $\mr_{\Sp^n}(P, a)(x)$ obtains minimum at $p_i$. Furthermore, according to the equality
    \[\pi = \mr_{\Sp^n}(P, a)(p_i) + \mr_{\Sp^n}(P, a)(\bar{p_i}),\]
    we deduce that $\mr_{\Sp^n}(P, a)(x)$ obtains maximum at $\bar{p_i}$. Therefore, the function $d_{\Sp^n}(\cdot, P)$ obtains local maximum at $\bar{p_i}$ as well. By Proposition \ref{proposition:held_point_as_local_max}, we know that $p_i$ is held by $P$. This shows that $P$ is pointwise extremal.
\end{proof}

On the other hand, for every pointwise extremal configuration on $\Sp^1$, we get a function 
inside the $\TS(\Sp^1)$ from the mountain range function construction, see Remark \ref{rmk:odd-n-gon} below. However, it is not the case for higher-dimensional spheres.

\begin{remark}\label{rmk:odd-n-gon}
    Let $P$ be the vertex set of an inscribed regular odd $n$-gon in $\Sp^1$. Then it is clear that the function
    \[\mr_{\Sp^n}\left(P, \frac{1}{2}\diam(P)\right)(\theta) = d_{\Sp^1}(\theta,P) + \frac{1}{2}\diam(P)\]
 is $1$-Lipschitz and satisfies the property that the sum of function values on any antipodal pair is $\pi$ and  therefore by Lemma \ref{lm:antipodal} this function is in $\TS(\Sp^1)$.
\end{remark}

\begin{remark}
The above mountain range function construction for a pointwise extremal configurations on higher dimensional spheres does not necessarily produce a function in the tight span. In fact, let $P$ be a subset of $\Sp^n$ such that the mountain range function
    \[\mr_{\Sp^n}\left(P, \frac{1}{2}\diam(P)\right)(x) = d_{\Sp^n}(x, P) + \frac{1}{2}\diam(P)\]
    is inside $\TS(\Sp^n)$. If $n\geq 2$ then $P$ must contain infinitely many points. First we observe that if $\diam(P) = \pi$ then $\mr_{\Sp^n}(P, \frac{1}{2}\diam(P))(x)\geq \frac{\pi}{2}$ for every $x$ and by Lemma \ref{lm:antipodal}, we have $d_{\Sp^n}(x,P) = 0$ for every $x$ and hence, $P$ is the whole sphere $\Sp^n$. Therefore, for other cases, we can assume $P$ does not contain a pair of antipodal points. Fix a point $p_i\in P$ and let $\bar{p_i}$ be its antipodal point. If $P$ is finite, then there exists some $\epsilon>0$ such that for every point $x\in B_\epsilon(p_i,\Sp^n)$, $d_{\Sp^n}(x,P)=d_{\Sp^n}(x, p_i)$. Note that, for $\mr_{\Sp^n}(P, \frac{1}{2}\diam(P))$ to be a function in the tight span $\TS(\Sp^n)$, it must satisfy
    \begin{itemize}
        \item $\mr_{\Sp^n}(P, \frac{1}{2}\diam(P))(p_i) + \mr_{\Sp^n}(P, \frac{1}{2}\diam(P))(\bar{p_i}) = \pi$
        \item $\mr_{\Sp^n}(P, \frac{1}{2}\diam(P))(x) + \mr_{\Sp^n}(P, \frac{1}{2}\diam(P))(\bar{x}) = \pi$
    \end{itemize}
    From our assumption on $x$, we have
    $$\mr_{\Sp^n}\left(P, \frac{1}{2}\diam(P)\right)(x) = \mr_{\Sp^n}\left(P, \frac{1}{2}\diam(P)\right)(p_i) + d_{\Sp^n}(x, p_i).$$ Therefore,
    \begin{align*}
        d_{\Sp^n}(\bar{x}, P) &= \mr_{\Sp^n}\left(P,\frac{1}{2}\diam(P)\right)(\bar{x}) - \frac{1}{2}\diam(P) \\
        & = \mr_{\Sp^n}\left(P, \frac{1}{2}\diam(P)\right)(\bar{p_i}) - d_{\Sp^n}(x, p_i)- \frac{1}{2}\diam(P)\\
        & = \mr_{\Sp^n}\left(P, \frac{1}{2}\diam(P)\right)(\bar{p_i}) - d_{\Sp^n}(\bar{x}, \bar{p_i})- \frac{1}{2}\diam(P)\\
        & = d_{\Sp^n}(\bar{p_i}, P) - d_{\Sp^n}(\bar{x}, \bar{p_i})
    \end{align*}
    
    This implies that there must be a point $p_{i'}$ in $\comax_P(p_i)$ such that $\bar{x}, \bar{p}_i, p_{i'}$ lie on a geodesic of $\Sp^n$. When $n\geq 2$, there are infinitely many geodesic circles based at $\bar{p}_i$ such that pairwise intersections are exactly $\{p_i, \bar{p}_i\}$. Therefore, by varying $\bar{x}\in B_\epsilon(P)$, there must be infinitely many different $p_{i'}$ for different choices of $x$ and thus $P$ consists of infinitely many points.
\end{remark}

Nevertheless, the following rotated odd $n$-gon on $\Sp^n$ gives a function in the tight span through the mountain range function construction. Let us introduce a convenient notion for subsets of $\Sp^n$ that give rise to an element in $\TS(\Sp^n)$ through the mountain range function construction.
\begin{definition}
    Let $P$ be a subset of $\Sp^n$, if the mountain range function
    \[\mr_{\Sp^n}\left(P, \frac{1}{2}\diam(P)\right)(x) = d_{\Sp^n}(x,P) + \frac{1}{2}\diam(P)\]
    is inside $\TS(\Sp^n)$, we then say $P$ is an \emph{admissible subset} in $\Sp^n$.
\end{definition}

The following example shows that we can get admissible subsets of $\Sp^n$ by iteratively revolving odd $n$-gon on $\Sp^1$.

\begin{example}[Examples of Mountain range functions inside $\TS(\Sp^n)$]\label{ex:MTrginTSSPn}
For every positive integer $m$, the regular $(2m+1)$-gon $P_m^1$ and the constant $\frac{1}{2}\diam(P_m^1)=\frac{m\pi}{2m+1}$ give rise to a mountain range function $\mathrm{MR}(P_m^1,\frac{m\pi}{2m+1})\in \TS(\Sp^1)$. Based on $P_m^1$, we will introduce $P_m^n\subseteq \Sp^n$ through rotations for every integer $n>1$ such that $\mathrm{MR}(P_m^n,\frac{m\pi}{2m+1})\in \TS(\Sp^n)$. Consider the spherical coordinate for $\Sp^n\subseteq\R^{n+1}$ as follows
\begin{align*}
    &x_1=\cos(\varphi_1)\\
    &x_2=\sin(\varphi_1)\cos(\varphi_2)\\
    &x_3=\sin(\varphi_1)\sin(\varphi_2)\cos(\varphi_3)\\
    &\cdots\\
    &x_{n}=\sin(\varphi_1)\sin(\varphi_2)\cdots\sin(\varphi_{n-1})\cos(\varphi_n)\\
    &x_{n+1}=\sin(\varphi_1)\sin(\varphi_2)\cdots\sin(\varphi_{n-1})\sin(\varphi_n)
\end{align*}
where $\varphi_1\in[0,\pi]$ and $\varphi_k\in[0,2\pi)$ for all $k=2,\ldots,n$. Assume that $P_m^1$ is represented by polar coordinates $\left\{\frac{2k\pi}{2m+1}|\,k=0,1,\ldots,2m+1\right\}$. We let 
$$Q_m\coloneqq \left\{\frac{2k\pi}{2m+1}|\,k=0,1,\ldots,2m+1\right\}\cap[0,\pi]=\left\{\frac{2k\pi}{2m+1}|\,k=0,1,\ldots,m+1\right\}.$$
Then, we let $P_m^n\coloneqq Q_m\times [0,2\pi]^{n-1}$ denote a set of points in $\Sp^n$ in terms of spherical coordinate. For every point $x\in\Sp^n$ represented by spherical coordinate $(\psi_1,\ldots,\psi_n)$, consider the great circle $\Sp^1_x$ containing $x$ and $(0,\ldots,0)$ represented by the following set of equations: $(\varphi_1,\ldots,\varphi_n)\in \Sp^1_x$ iff $\varphi_k=\psi_k$ for all $k=2,\ldots,n$. Of course, $(0,\psi_2,\ldots,\psi_n)$ and $(0,\ldots,0)$ represent the same point in $\Sp^n$. Then, $\Sp_x^1\cap P_m^n=\{(\frac{2k\pi}{2m+1},\psi_2,\ldots,\psi_n)|\,k=0,1,\ldots,m+1\}$ is itself a regular $(2m+1)$-gon, denoted by $P_{m,x}^1$. Since the antipodal point $\bar{x}\in \Sp_x^1$, we have that
\begin{align*}
    &\mathrm{MR}\left(P_m^n,\frac{m\pi}{2m+1}\right)(x)+\mathrm{MR}\left(P_m^n,\frac{m\pi}{2m+1}\right)(\bar{x})\\
    =&\mathrm{MR}\left(P_{m,x}^1,\frac{m\pi}{2m+1}\right)(x)+\mathrm{MR}\left(P_{m,x}^1,\frac{m\pi}{2m+1}\right)(\bar{x})\\
    =&\pi,
\end{align*}
where the last equality follows from $\mathrm{MR}\left(P_{m,x}^1,\frac{m\pi}{2m+1}\right)\in\TS(\Sp^1_x)$ and Lemma \ref{lm:antipodal}. It is obvious that $\mathrm{MR}\left(P_m^n,\frac{m\pi}{2m+1}\right)$ is 1-Lipschitz. Then, by Lemma \ref{lm:antipodal} again, we have that $\mathrm{MR}\left(P_m^n,\frac{m\pi}{2m+1}\right)\in \TS(\Sp^n)$.
\end{example}

\begin{figure}
    \centering
    \includegraphics[width=0.6\textwidth]{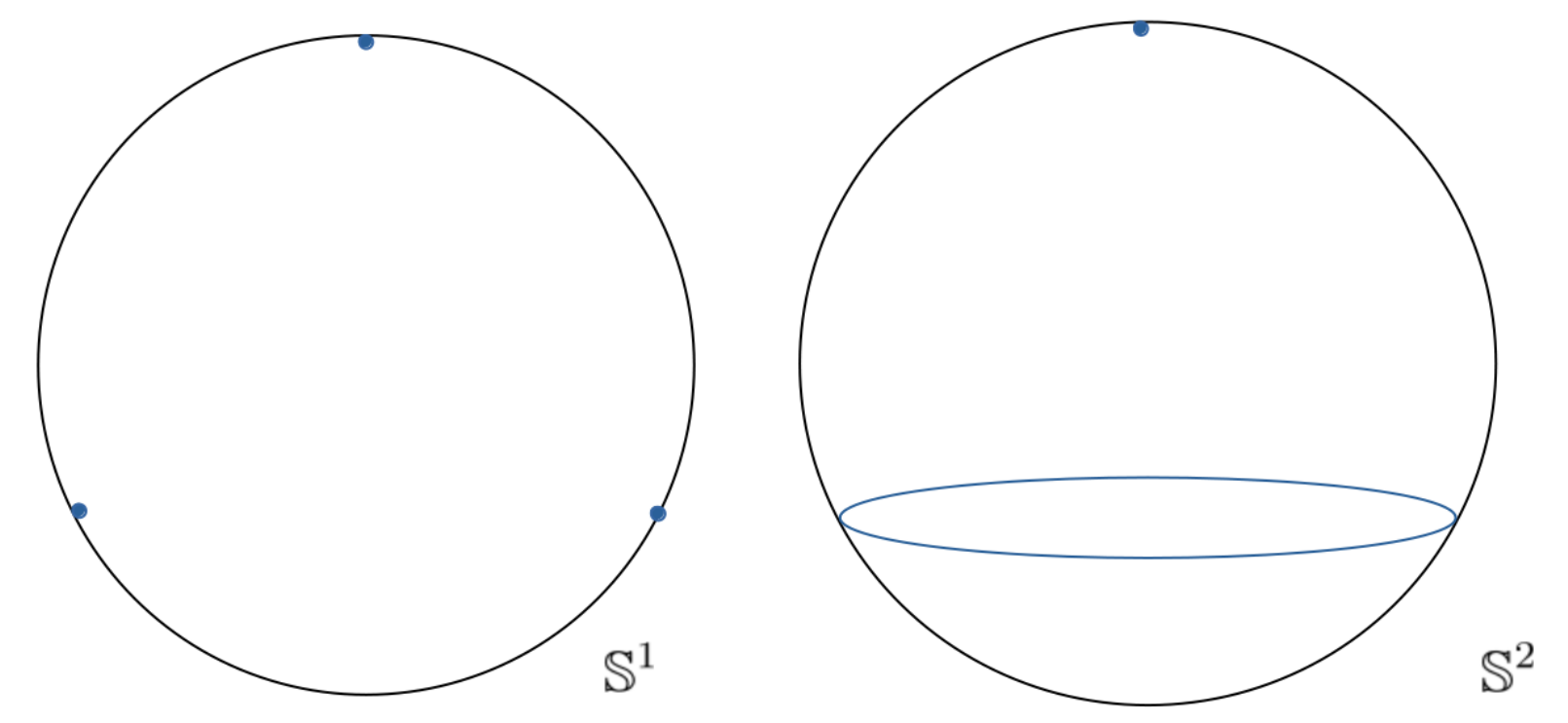}
    \caption{Example of revolved admissible set: on the left, we have the three vertices of an inscribed equilateral triangle of $\Sp^1$; on the right, we have the revolved configuration on $\Sp^2$, the north pole along with a circle.}
    \label{fig:revolved_asd}
\end{figure}

In general, we have the following result regarding revolved admissible sets.
\begin{proposition}
    Let $P_{n+1}$ be a subset of $\Sp^{n+1}$ such that there exists a point $a\in P_{n+1}$ with the following properties,
    \begin{enumerate}
        \item any rotation at the point $a$ preserves $P_{n+1}$,
        \item for any $\Sp^n$ that contains both $a$ and $\bar{a}$, that is a slice under the rotation at $a$, $P_{n+1}\cap\Sp^n$ is admissible in $\Sp^n$.
    \end{enumerate}
    Then, $P_{n+1}$ is admissible in $\Sp^{n+1}$.
\end{proposition}
\begin{proof}
    Let $x$ be a point in $\Sp^{n+1}$, then there exists a $\Sp^n$ slice that contains $x$, $a$ and $\bar{a}$. For any point $p\in P_{n+1}$, the orbit of $p$ under the rotation intersects perpendicularly with the above chosen $\Sp^n$ at a point $p'$. Therefore, $d_{\Sp^{n}}(x, p') = d_{\Sp^{n+1}}(x, p')\leq d_{\Sp^{n+1}}(x, p)$ which further implies
    \[d_{\Sp^{n+1}}(x, P_{n+1}) = d_{\Sp^{n}}(x, P_{n+1}\cap \Sp^n).\]
    This allows us to show $\diam(P_{n+1}) = \diam(P_{n+1}\cap \Sp^n)$ where $\Sp^n$ is a slice under rotation. Note that, as rotations are isometries and $P_{n+1}$ is preserved under rotations, the quantity $\diam(P_{n+1}\cap \Sp^n)$ is independent of the choice of slice. For any point $p$ in $P_{n+1}$, the diameter of $P_{n+1}$ at $p$ is
    \[\max_{p'\in P_{n+1}}d_{\Sp^{n+1}}(p', p) = \pi - d_{\Sp^{n+1}}(\bar{p}, P_{n+1}) = \pi - d_{\Sp^n}(\bar{p}, P_{n+1}\cap \Sp^n)= \max_{p'\in  P_{n+1}\cap \Sp^n}d_{\Sp^n}(p', p).\]
    where $\Sp^n$ is a slice containing $p$ and hence, $\bar{p}$. As $p$ is arbitrary, we get $\diam(P_{n+1}) = \diam(P_{n+1}\cap \Sp^n)$. We then have the following calculation.
    \begin{align*}
          & \mr_{\Sp^{n+1}}\left(P_{n+1}, \frac{1}{2}\diam(P_{n+1})\right)(x) + \mr_{\Sp^{n+1}}\left(P_{n+1}, \frac{1}{2}\diam(P_{n+1})\right)(\bar{x})                                                                \\
        & =  d_{\Sp^{n+1}}(x, P_{n+1})(x) + \frac{1}{2}\diam(P_{n+1}) + d_{\Sp^{n+1}}(x, P_{n+1})(\bar{x})+ \frac{1}{2}\diam(P_{n+1})                                     \\
        & =  d_{\Sp^{n}}(x, P_{n+1}\cap \Sp^n)(x) + \frac{1}{2}\diam(P_{n+1}\cap \Sp^n) + d_{\Sp^{n}}(x, P_{n+1}\cap \Sp^n)(\bar{x})+ \frac{1}{2}\diam(P_{n+1}\cap \Sp^n) \\
        & =  \pi
    \end{align*}
    The last line comes from the fact that $P_{n+1}\cap \Sp^n$ is admissible and applying Lemma \ref{lm:antipodal}. As $\mr_{\Sp^{n+1}}\left(P_{n+1}, \frac{1}{2}\diam(P_{n+1})\right)$ is a $1$-Lipschitz function by Proposition \ref{prop:mr_in_D1}, the result then comes from the characterization of functions in the tight span as in Lemma \ref{lm:antipodal}.
\end{proof}


\section{The case of $\Sp^n$ with $\ell^\infty$-metric}\label{sec:TSSninfty}

Let $n$ be a positive integer. For simplicity of notation, let $d_\infty$ denote the metric induced by $\ell^\infty$-norm on $\R^n$, and we write $\R^n_\infty:=(\R^n,\ell^\infty)$.
We denote by $\Sp^n_\infty:=\{x\in \R^{n+1}: \norm{x}_2=1\}$ the $n$-dimensional sphere equipped with $\ell^\infty$-metric, and similarly $\bbD^{n+1}_\infty:=\{x\in \R^{n+1}: \norm{x}_2\leq 1\}$ the $(n+1)$-dimensional ball equipped with $\ell^\infty$-metric. 
In this section, we study the tight span of $\Sp^n_\infty$. 

\paragraph{The tight span of $\Sp^1_\infty$.}
We start with the case of $\Sp^1_\infty\subset \R^2_\infty$. Notice that $\R^2_\infty$ and the \emph{Manhattan plane} $(\R^2,\ell^1)$ are isometric, but for $n>2$, $\R^n_\infty$ and $(\R^n,\ell^1)$ are not isometric. In particular, while $\R^n_\infty$ is hyperconvex for every $n$, $(\R^n,\ell^1)$ is hyperconvex only when $n\leq 2.$ On the other hand, the $\ell_\infty$ and $\ell_1$ metrics on real spaces are closely related: \cite[Theorem 5]{herrlich1992hyperconvex} proves that the space $\R^{2^n-1}_\infty$ is isometric to the tight span of $(\R^n,\ell^1)$.

The tight span of a subset of $\R^2$ equipped with either the $\ell^\infty$ or the $\ell^1$-norm has been completely studied in \cite{eppstein2011optimally,kilicc2015tight,kilicc2016tight}, in which case the tight span is the same as the orthogonal convex hull. In addition, in \cite{kilicc2021algorithm}, the authors develop a procedure to compute the tight span of finite subsets of $(\R^2,\ell^1)$.

In $\R^2_\infty$, a subspace $Y$ is said to be geodesically convex if any two points in $Y$ are connected by a geodesic that is fully contained in $Y$.
By \cite[Theorem 2]{kilicc2016tight}, the tight span of $X\subseteq\R^2_\infty$ is isometric to any closed geodesically convex subspace $Y$ containing $X$, which is minimal (w.r.t. inclusion) with these properties. From this, we have:

\begin{corollary}
The tight span of $\Sp^1_\infty$ is isometric to $\bbD^2_\infty$.
\end{corollary}

Unlike the case of $\R^2_\infty$, when $n\geq 3$ the tight span of subsets of $\R^n_\infty$ appear difficult to describe. Some related results in the literature include studies of the injectivity of subsets of $\R^n_\infty$, such as \cite{pavon2016injective} for the injectivity of convex polyhedrons in $\R^n_\infty$ and \cite{descombes2017injective} for the injectivity of more general subsets of $\R^n_\infty$. In this section, we study the tight span of $\Sp^n_\infty$, as well as the injectivity of $\bbD^{n+1}_\infty$.

In Section \ref{sec:surrouning points}, for a given compact set $X\subseteq\R^n_\infty$ we introduce a set $\caF(X)\subseteq\R^n_\infty$ that is a sub-metric space of the tight span $\TS(X)$, and use it to show that $\bbD^{n+1}_\infty$ is not injective when $n\geq 2$. In Section \ref{sec:minimal points}, we identify an injective set $\caE(X)\subseteq\R^n_\infty$ containing $\TS(X)$, and prove that $\caE(\Sp^2_\infty)$ (in this case, equal to $\caF(\Sp^2_\infty)$) is isometric to $\TS(\Sp^2_\infty)$.

\subsection{The set of $X$-surrounding points $\caF(X)$} \label{sec:surrouning points}
In this section, we study a particular subspace $\caF(X)$ of the tight span $\TS(X)$, for a given compact set $X\subseteq\R^n_\infty$. We show that when $n\geq 2$, the tight span of $\Sp^n_\infty$ (or $\bbD^{n+1}_\infty$) properly contains $\bbD^{n+1}_\infty$ up to isometry.

As in \cite{descombes2017injective}, we divide $\R^n_\infty$ into the following cone-shaped regions. For each $1\leq i\leq n$, let $\Lambda_i :=\{x\in \R^n_\infty: x_i=\|x\|_{\infty}\}.$
For each $x\in \R^n_\infty$, $\xi\in\{\pm 1\}$ and $i=1,\dots,n$, let
\[x+\xi \Lambda_i  :=\{x+\xi z: z\in \Lambda_i \}=\{z\in \R^n_\infty: z_i-x_i = \xi\, d_\infty(z,x)\}.\]

Clearly, $\R^n_\infty = \cup_{\xi\in\{\pm 1\},i\in\{1,\dots,n\}}(x+\xi \Lambda_i).$ See Figure \ref{fig:regions} for illustrations of $\xi \Lambda_i $ in $\R^2_\infty$ and $\R^3_\infty$.

\begin{figure}[h!]
    \centering

\begin{tikzpicture}[scale=0.8]
\coordinate (b) at (0,0); 
\coordinate (c) at (0,2); 
\coordinate (f) at (2,2);
\coordinate (g) at (2,0); 

\draw[black, dashed] (b) -- (c) -- (f) -- (g) -- cycle;
\draw[blue] (b) -- (f);
\draw[blue] (c) -- (g);

\node[mark=none] at (-2.5,1){$\text{In }\R^2:$};
\node[mark=none, blue] at (2.4,1){$\Lambda_1$};
\node[mark=none, blue] at (-0.6,1){$-\Lambda_1$};
\node[mark=none, blue] at (1,2.3){$\Lambda_2$};
\node[mark=none, blue] at (1,-.5){$-\Lambda_2$};
\end{tikzpicture}

\newcommand{\Depth}{2}
\newcommand{\Height}{2}
\newcommand{\Width}{2}
\begin{tikzpicture}[scale=0.8]
\coordinate (O) at (0,-0.25,0);
\coordinate (A) at (0,\Width-0.25,0);
\coordinate (D) at (\Depth,-0.25,0); 
\coordinate (E) at (\Depth,\Width-0.25,0); 

\coordinate (B) at (0,\Width,\Height);
\coordinate (C) at (0,0,\Height); 
\coordinate (F) at (\Depth,\Width,\Height);
\coordinate (G) at (\Depth,0,\Height);

\coordinate (M) at (0.5\Depth,0.5\Width,0,5\Height); 

\draw[black, dashed] (A) -- (B) -- (F) -- (E) -- cycle;
\draw[black, dashed] (C) -- (B) -- (F) -- (G) -- cycle;
\draw[black, dashed] (D) -- (E) -- (F) -- (G) -- cycle;
\draw[black, dashed, dashed] (D) -- (O) -- (A);
\draw[black, dashed, dashed] (C) -- (O);

\draw[blue] (E) -- (C);
\draw[blue] (O) -- (F);
\draw[blue] (A) -- (G);
\draw[blue] (D) -- (B);

\node[mark=none] at (-3,0.5){$\text{In }\R^3:$};
\node[mark=none, blue] at (2.4,0.6){$\Lambda_1$};
\node[mark=none, blue] at (-1.3,0){$-\Lambda_1$};
\node[mark=none, blue] at (0.5,2){$\Lambda_3$};
\node[mark=none, blue] at (0.25,-1.3){$-\Lambda_3$};
\node[mark=none, blue] at (2.1,1.9){$\Lambda_2$};
\node[mark=none, blue] at (-1,-1){$-\Lambda_2$};
\end{tikzpicture}

    \caption{The regions $\xi\Lambda_i$ in $\R^2_\infty$ and $\R^3_\infty$.}
    \label{fig:regions}
\end{figure}
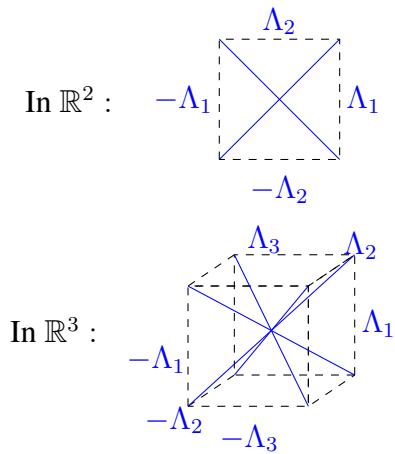

\begin{proposition} \label{prop:property of C_i and I_xy1}
Let $x,y\in \R^n_\infty$. for every $i\in\{1,\dots,n\}$ and $\xi\in\{\pm 1\}$,

\begin{enumerate}
    \item \label{prop:y in x+Lambda_i  iff x in y-Lambda_i} $y\in x+\xi \Lambda_i  \iff x\in y-\xi \Lambda_i .$
    \item \label{prop:C_y subset C_x} $y\in x+\xi \Lambda_i  \implies (y+\xi \Lambda_i )\subseteq (x+\xi \Lambda_i )$ and $(x-\xi \Lambda_i )\subseteq (y-\xi \Lambda_i )$.
\end{enumerate}
\end{proposition}
\begin{proof}
Part (\ref{prop:y in x+Lambda_i  iff x in y-Lambda_i}) is trivial. Part (\ref{prop:C_y subset C_x}) is true, because for $y\in x+\xi \Lambda_i $ and any $z\in y+\xi \Lambda_i $, we have
\[d_\infty(z,x)\geq \xi(z_i-x_i)=\xi(z_i-y_i)+\xi(y_i-x_i)= d_\infty(z,y)+d_\infty(y,x) \geq d_\infty(z,x)\]
which implies $\xi(z_i-x_i)=d_\infty(z,x)$, i.e. $y+\xi \Lambda_i\subseteq x+\xi \Lambda_i $. The other inclusion is proved in a similar way.
\end{proof}

\begin{definition}[$X$-surrounding points] 
For a compact subset $X\subseteq\R^n_\infty$, we call a point $p\in\R^n_\infty$ a \emph{$X$-surrounding point} if $(p+\xi  \Lambda_i)\cap X\neq \emptyset$ for every $i\in \{1,\dots,n\}$ and $\xi\in\{\pm 1\}$.

Denote $\caF(X):=\{ p\in \R^n_\infty: p\text{ is a }X\text{-surrounding point}\}$, and call it the \emph{set of $X$-surrounding points}.
\end{definition}

Clearly, $X\subseteq\caF(X)$, and $\caF(X)\subseteq\caF(Y)$ whenever $X\subseteq Y\subseteq\R^n_\infty$.\vspace{\baselineskip}

For each $z\in \R^n_\infty$, let $d_\infty^X(z,\cdot):=d_\infty(z,\cdot)\vert_X\in L^\infty(X)$. We then have the following map:
\begin{align*}
    \iota_X:\R^n_\infty&\longrightarrow L^\infty(X)\\
    z&\longmapsto d_\infty^X(z,\cdot)
\end{align*}
By definition, the set $\iota_X^{-1}(\TS(X))$ consists of all $z\in\R^n_\infty$ such that the function $d_\infty^X(z,\cdot)$ is minimal (see Definition \ref{def:inj}).

\begin{proposition}\label{prop:FX_EX_caEX}
Let $X$ be a compact subset of $\R^n_\infty$. Then, for every $X$-surrounding point $p$, $d_\infty^X(p,\cdot)$ is minimal. i.e., $\iota_X(\caF(X))\subseteq\TS(X)$. In addition, the map $\iota_X:\caF(X)\to \TS(X)$ preserves distances. Thus, $\caF(X)\preceq\TS(X)$.
\end{proposition}
\begin{proof} 
Fix an arbitrary $p\in \caF(X)$ and arbitrary $x\in X$. Then, $x\in p+\xi  \Lambda_i$ for some $i\in \{1,\dots, n\}$ and $\xi\in \{\pm 1\}$. Take $y\in (p-\xi  \Lambda_i)\cap X\neq\emptyset$ (because $p\in \caF(X)$). Notice that
\[d_\infty(x,y)\leq d_\infty^X(p,x)+d_\infty^X(p,y)=\xi(x_i-p_i)+(-\xi)(y_i-p_i)=\xi(x_i-y_i)\leq d_\infty(x,y).\]
Hence, $d_\infty^X(p,x)+d_\infty^X(p,y)=d_\infty(x,y)$ so that $d_\infty^X(p,\cdot)$ is minimal. i.e., $d_\infty^X(p,\cdot)\in \TS(X)$.\vspace{\baselineskip}

We now prove that $\iota_X:\caF(X)\to \TS(X)$ preserves distances. Fix arbitrary $p,q\in \caF(X)$. One can assume that $q\in p+\xi \Lambda_i$ for some $i\in \{1,\dots,n\}$ and $\xi\in \{\pm 1\}$. Because $q\in\caF(X)$, $(q+\xi\Lambda_i)\cap X\neq \emptyset$. Take $x\in(q+\xi \Lambda_i)\cap X \subseteq (p+\xi \Lambda_i)\cap X$ where the last inclusion holds by the item (2) of Proposition \ref{prop:property of C_i and I_xy1}. Therefore,
\[d_\infty^X(p,x)-d_{\infty}^X(q,x)=\xi(x_i-p_i)-\xi(x_i-q_i)=\xi(q_i-p_i)=d_\infty(p,q).\]
Thus, $\|d_\infty^X(p,\cdot)-d_\infty^X(q,\cdot)\|_\infty\geq d_\infty(p,q)$. The reverse inequality $\|d_\infty^X(p,\cdot)-d_\infty^X(q,\cdot)\|_\infty\leq d_\infty(p,q)$ follows directly from the triangle inequality. Therefore, $\iota_X:\caF(X)\to \TS(X)$ preserves distances, implying that $\caF(X)\preceq\TS(X)$.
\end{proof}

\begin{example}
Notice that $\caF(X)$ is not necessarily injective. Consider the example $X:={(0,3)}\cup (\{0\}\times [-1,1])\subseteq \R^2_\infty$, for which we have $X\subsetneq \caF(X)\prec \TS(X)$. See Figure \ref{fig:caFX v.s. EX}. 

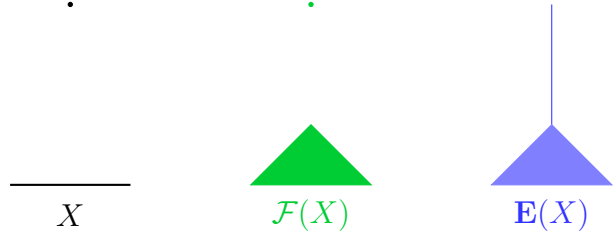
\begin{figure}[h!]
    \centering
    \begin{tikzpicture}[scale=0.8]
\node (a) at (-1,0) {};
\node (c) at (0,3) {};
\node (e) at (1,0) {};

\node (d1) at (4,1) {};
\node (a1) at (3,0) {};
\node (c1) at (4,3) {};
\node (e1) at (5,0) {};

\node (a2) at (7,0) {};
\node (c2) at (8,3) {};
\node (e2) at (9,0) {};
\node (d2) at (8,1) {};

\filldraw (c) circle[radius=1pt];
\draw [thick] (e.center)--(a.center) node at (0,-0.5) {$X$}; 

\filldraw [color= green!80!blue] (c1) circle[radius=1pt];
\node[color= green!80!blue] at (4,-0.5) {$\caF(X)$}; 
\filldraw [color=green!80!blue] (a1.center) -- (d1.center) -- (e1.center) -- cycle;

\node[color= blue!80] at (8,-0.5) {$\TS(X)$}; 
\draw [color=blue!80] (d2.center)--(c2.center); 
\filldraw [color=blue!50] (a2.center) -- (d2.center) -- (e2.center) -- cycle;
\end{tikzpicture} 
    \caption{An example such that $X\subsetneq \caF(X)\prec \TS(X)$.}
    \label{fig:caFX v.s. EX}
\end{figure}
\end{example}

\begin{remark} The set $\caF(X)$ is injective if and only if $\caF(X)\cong \TS(X)$. 
\end{remark}

\begin{proposition}
If $X\subseteq\R^n_\infty$ is compact, then $\caF(X)$ is compact.
\end{proposition}
\begin{proof} Because $X$ is bounded, we can choose a large enough bounded $n$-dimensional cube $C=[-a,a]^n\supseteq X$ for some positive number $a$.
For any $p\notin C$, there exists some $1\leq i\leq n$ such that $|p_i|> a$. Without loss of generality, assume $p_i$ is positive. Any point $x\in p+\Lambda_i$ satisfies $x_i>p_i>a$, and thus cannot be in $C$. Therefore, $(p+\Lambda_i)\cap C=\emptyset$, implying that $p\notin \caF(X)$. Thus, $\caF(X)$ is contained in $ C$ and is bounded.

Let $p\in \R^n_\infty\backslash\caF(X)$ and assume that $ \xi\in\{\pm 1\}$ and $1\leq i\leq n$ are such that $ (p+\xi\Lambda_i)\cap X=\emptyset$. Since $X$ and $p+\xi\Lambda_i$ are both closed, $r:=d_\infty (X,p+\xi\Lambda_i)>0$. For any $q\in B_r(p,\R^n_\infty)$, we claim that $q+\xi\Lambda_i\subseteq B_r(p+\xi\Lambda_i,\R^n_\infty)$. Indeed, for every $z\in q+\xi\Lambda_i$, we have $z+p-q\in (q+p-q)+\xi\Lambda_i=p+\xi\Lambda_i$, and it follows from
\[d_\infty(z,p+\xi\Lambda_i)\leq d_\infty(z,z+p-q)=d_\infty(p,q)<r\]
that $z\in B_r(p+\xi\Lambda_i,\R^n_\infty)$. 
Therefore, $B_r(p,\R^n_\infty)\subseteq \R^n_\infty\backslash\caF(X)$, and thus $\caF(X)$ is closed. Hence, $\caF(X)$ is compact.
\end{proof}

\begin{proposition}
If $X\subseteq\R^n_\infty$ is convex, then $\caF(X)$ is convex.
\end{proposition}
\begin{proof}
For every two points $p,q\in \caF(X)$ and each $\lambda\in [0,1]$, let $w:=\lambda p+(1-\lambda) q$.
By the definition of $\caF(X)$, there exist $x\in (p+\Lambda_1)\cap X$ and $y\in (q+\Lambda_1)\cap X$.
Because $X$ is convex, we have $z:=\lambda x+(1-\lambda)y\in X$. It follows from
\[d_\infty(w,z)\leq \lambda d_\infty(p,x)+(1-\lambda)d_\infty(q,y)=\lambda (x_1-p_1)+(1-\lambda)(y_1-q_1)=z_1-w_1\leq d_\infty(w,z)\]
that $z\in w+\Lambda_1$. Therefore, $(w+\Lambda_1)\cap X\neq \emptyset$, and by a similar argument we have $(w+\xi\Lambda_i)\cap X\neq \emptyset$ for every $\xi\in\{\pm 1\}$ and $1\leq i\leq n$. Thus, $\caF(X)$ is convex.
\end{proof}

Next, we show that the tight span of the $n$-sphere $\Sp ^n_\infty$ (or the $(n+1)$-ball $\bbD^{n+1}_\infty$) with $\ell^\infty$ metric properly contains the $(n+1)$-ball $\bbD^{n+1}_\infty$ for $n\geq 2$.

\begin{theorem}\label{thm:ESn_infinity}
Let $n\geq 2$. For the tight spans of $\Sp ^n_\infty$ and $\bbD^{n+1}_\infty$, we have
\begin{itemize}
    \item $\TS(\Sp ^n_\infty)\succeq \caF(\Sp^n_\infty)\supsetneq \bbD^{n+1}_\infty;$ 
    
    \item $\TS(\bbD^{n+1}_\infty)\succeq\caF(\bbD^{n+1}_\infty)\supseteq\caF(\Sp^n_\infty)\supsetneq \bbD^{n+1}_\infty.$
\end{itemize}
 As a consequence, $\bbD^{n+1}_\infty$ is not injective.
\end{theorem}
\begin{proof}  
To prove $\caF(\Sp^n_\infty)\supsetneq \bbD^{n+1}_\infty$ it is enough to find a point $p\in \caF(\Sp ^n_\infty)-\bbD^{n+1}_\infty$.

Let $x=\frac{1}{\sqrt{n+1}}(1,\dots,1)\in \R^{n+1}_\infty$. Let $p=(1+\lambda)x$, where $\lambda$ is chosen to be a positive number such that $\lambda^2+2\lambda\leq \frac{(n-1)^2}{4n}$. It is clear that $p\notin\bbD^{n+1}_\infty$. See Figure \ref{fig:ESn} for the case of $n=2$.
\begin{figure}[h!]
    \centering
    \begin{tikzpicture}
    \node[fill=black,circle,inner sep=.1em] (o) at (0,0)  {};
    \filldraw[fill=none](0,0) circle (1.5);
    \draw [dashed](-1.5,0) to[out=90,in=90,looseness=.5] (1.5,0);
    \draw (1.5,0) to [out=270,in=270,looseness=.5] (-1.5,0);
    \draw (0,1.5) to [out=-10,in=90,looseness=.5] (0.8,-0.38);
    
    \node[fill=black,circle,inner sep=.1em] (x) at (0.55,0.7)  {};
    \node[fill=red,circle,inner sep=.1em] (p) at (0.7,0.891) {};
    \node[right, red] at (p)  {\footnotesize $p=\frac{1+\lambda}{\sqrt{3}}(1,1,1)$};
    \node[left] at (x)  {\footnotesize $\frac{1}{\sqrt{3}}(1,1,1)$};
    \draw [dashed] (0,0) to (x);
    \draw (x) to (p);
	\end{tikzpicture}
    \caption{A point $p$ that is in $\caF(\Sp^2_\infty)\cong \TS(\Sp ^2_\infty)$ but not in $\bbD^{3}_\infty$, see Theorem \ref{thm:ESn_infinity}.}
    \label{fig:ESn}
\end{figure}
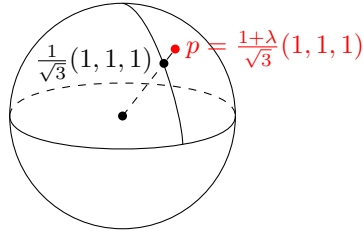

Let $v^1=\frac{1}{\sqrt{n+1}}(-1,1,\dots,1), v^2=\frac{1}{\sqrt{n+1}}(1,-1,1,\dots,1),\dots,v^{n+1}=\frac{1}{\sqrt{n+1}}(1,\dots,1,-1)$ be points in $\Sp ^n_\infty$. Fix arbitrary $i=1,\cdots,n+1$. It follows from the fact $p_i-v_i^i=d_\infty(p,v^i)$ that $v^i\in p- \Lambda_i$. Thus, $(p- \Lambda_i)\cap \Sp ^n_\infty\neq \emptyset$. Next, we prove $(p+\Lambda_i)\cap \Sp ^n_\infty\neq \emptyset$. For any $t>0$, the point $-tv^i$ is in $\Lambda_i$, implying that $p-t v^i\in p+\Lambda_i$. We compute
\begin{align*}
 \|p-t v^i\|_{2}^2 - 1 &= \lbracket \frac{1+\lambda+t}{\sqrt{n+1}}  \rbracket^2 + n\lbracket \frac{1+\lambda-t}{\sqrt{n+1}}\rbracket^2 -1 \\
 &= t^2-\frac{2(n-1)(1+\lambda)}{n+1} t+ (\lambda^2+2\lambda)
\end{align*}
Because $\lambda$ is positive and $n\geq 2$, the above quadratic function has a positive axis of symmetry. Also, the discriminant of this function $(n-1)^2(1+\lambda)^2-(n+1)^2(\lambda^2+2\lambda)$ is nonnegative because of the assumption on $\lambda$. Hence, this function obtains zero for some $t>0$. Thus, we have proved that there is some $t>0$ such that $p-t v^i\in \Sp ^n_\infty$, i.e. $(p- \Lambda_i)\cap \Sp ^n_\infty\neq \emptyset$. Therefore, $p\in \caF(\Sp ^n_\infty)-\bbD^{n+1}_\infty$ as we required.
\end{proof}

\begin{remark}
Let us see why the above proof of Theorem \ref{thm:ESn_infinity} does not go through when $n = 1$. We use the same notation as the theorem. When $\lambda>0$, $\|p-t v^i\|_{2}^2 - 1 = t^2 + (\lambda^2+2\lambda)=0$ does not have a solution for $t$. Thus, $p=\frac{1+\lambda}{\sqrt{2}}(1,1)\notin \caF(\Sp_\infty^1)$ for every $\lambda>0$, unlike the case when $n\geq 2.$
\end{remark}

\subsection{The set of $X$-minimal points $\caE(X)$}
\label{sec:minimal points}

In this section, we identify an injective set $\caE(X)$ containing a given set $X\subseteq\R^n_\infty$, which thus also contains the tight span of $X$ up to isometry.

For every $x,y\in \R^n$, the \emph{closed metric interval} (see \cite[page 13]{deza2009encyclopedia}) between $x$ and $y$ is defined as the set
\[I_{xy}:=\{z\in \R^n: d_\infty(x,z)+d_\infty(z,y)=d_\infty(x,y)\}.\]

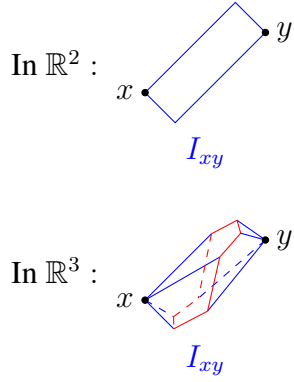
\begin{figure}[h!]
    \centering

\begin{tikzpicture}[scale=0.8]

\node[mark=none] at (2.5,1){$\text{In }\R^2:$};

\coordinate (x) at (4,0.5); 
\coordinate (y) at (6,1.5); 
\coordinate (f) at (5.5,2);
\coordinate (f1) at (4.5,0);

\draw[blue] (x) -- (f) -- (y) -- (f1) -- cycle;
\filldraw[black] (x) circle (1.5pt) node[anchor=east] {$x$};
\filldraw[black] (y) circle (1.5pt) node[anchor=west] {$y$};
\node[mark=none, blue] at (5,-0.5){$I_{xy}$};
\end{tikzpicture}

\hspace{2em}
\newcommand{\Depth}{2}
\newcommand{\Height}{2}
\newcommand{\Width}{2}

\begin{tikzpicture}[scale=0.8]
\coordinate (x) at (0,0); 
\coordinate (y) at (2,1); 

\draw [blue] (x) -- (45:3.7em) coordinate (a);
\draw [blue] (x) -- (30:3.4em) coordinate (b);
\draw [blue, dashed] (x) -- (-30:1.3em) coordinate (c);
\draw [blue] (x) -- (-45:1.6em) coordinate (d);

\draw [red] (a) -- (41:4.8em) coordinate (A);
\draw [red] (b) -- (35:4.6em) coordinate (B);
\draw [red, dashed] (c) -- (2:2.1em) coordinate (C);
\draw [red] (d) -- (-10:2.5em) coordinate (D);

\draw [blue] (y) -- (A);
\draw [blue] (y) -- (B);
\draw [blue, dashed] (y) -- (C);
\draw [blue] (y) -- (D);

\draw [red, dashed] (c) -- (d);
\draw [red] (A) -- (B);
\draw [red, dashed] (a) -- (C);
\draw [red] (b) -- (D);

\node[mark=none] at (-1.5,0.5){$\text{In }\R^3:$};
\filldraw[black] (x) circle (1.5pt) node[anchor=east] {$x$};
\filldraw[black] (y) circle (1.5pt) node[anchor=west] {$y$};
\node[mark=none, blue] at (1,-1){$I_{xy}$};
\end{tikzpicture}
    \caption{Examples of the regions $I_{xy}$ in $\R^2$ and $\R^3$. In both figures, the interiors are included.}
    \label{fig:regions2}
\end{figure}
See Figure \ref{fig:regions2} for illustrations of $I_{xy}$ in both $\R^2$ and $\R^3$, and below for some property of $I_{xy}$.

\begin{proposition} \label{prop:property of C_i and I_xy}
Let $x,y\in \R^n_\infty$.
\begin{enumerate}
    \item \label{prop:I_xy decomp} Suppose that $y\in x+\xi \Lambda_i $. Then, $I_{xy}= (x+\xi \Lambda_i )\cap (y-\xi \Lambda_i ).$ 
    
    \item \label{prop:I_xy contains D_xz} For every $z\in I_{xy}$, $I_{xz}\subseteq I_{xy}$.
\end{enumerate}
\end{proposition}
\begin{proof} For Part (\ref{prop:I_xy decomp}), notice that for any $z\in (x+\xi \Lambda_i )\cap (y-\xi \Lambda_i )$,
\begin{align*}
 d_\infty(x,z)+d_\infty(z,y)=\xi(z_i-x_i)+\xi(y_i-z_i)=\xi(y_i-x_i)=d_\infty(x,y).
\end{align*}
Hence, $z\in I_{xy}$ so that $(x+\xi \Lambda_i )\cap (y-\xi \Lambda_i )\subseteq I_{xy}$.

Next, for any $z\in I_{xy}$, we have
\begin{align*}
 d_\infty(x,z)+d_\infty(z,y)= d_\infty(x,y)=\xi(y_i-x_i)=\xi(y_i-z_i)+\xi(z_i-x_i)\leq d_\infty(x,z)+d_\infty(z,y)
\end{align*}
implying $\xi(y_i-z_i)=d_\infty(z,y)$ and $\xi(z_i-x_i)=d_\infty(x,z).$ Hence, $z\in (x+\xi \Lambda_i )\cap (y-\xi \Lambda_i )$ so that $I_{xy}\subseteq (x+\xi \Lambda_i )\cap (y-\xi \Lambda_i )$.\vspace{\baselineskip}

For Part (\ref{prop:I_xy contains D_xz}), without loss of generality one can assume that $y\in x+\xi \Lambda_i $ for some $\xi$ and $i$. Then for every $z\in I_{xy}=(x+\xi \Lambda_i )\cap (y-\xi \Lambda_i )$ (by Part (\ref{prop:I_xy decomp})), we have
\begin{align*}
    I_{xz}&=(x+\xi \Lambda_i )\cap (z-\xi \Lambda_i ), \text{ (because $z\in x+\xi \Lambda_i $)}\\
    &\subseteq  (x+\xi \Lambda_i )\cap (y-\xi \Lambda_i ), \text{ (because $z\in y-\xi \Lambda_i $ and Proposition \ref{prop:property of C_i and I_xy1})}\\ 
    &=I_{xy}.
\end{align*}
\end{proof}

\begin{definition}[$X$-minimal points] \label{def:X-minimal}
For a compact subset $X\subseteq\R^n_\infty$, we say that $z\in \R^n_\infty$ is an \emph{$X$-minimal point (in $\R^n_\infty$)}, if the distance function $d_\infty^X(z,\cdot):X\to \R_{\geq 0}$ is minimal. i.e., for every $x\in X,$ there exists $y\in X$ such that $d_\infty(x,z)+d_\infty(z,y)=d_\infty(x,y)$.

Denote by $\caE(X)$ the \emph{set of $X$-minimal points (in $\R^n_\infty$)}, and equip it with the $\ell^\infty$ metric. Equivalently, \[\caE(X)=\bigcap_{x\in X}\bigcup _{y\in X} I_{xy}.\]
\end{definition}

\begin{proposition} \label{prop:property of E}  Consider a subset $X$ of $\R^n_\infty$. Then, the following properties hold:
\begin{enumerate}
    \item \label{prop-part:X subset EX} $X\subseteq \caE(X)$.
    
    \item \label{prop-part:E^2=E}  For every $\tilde{X}$ such that $X\subseteq \tilde{X}\subseteq \caE(X)$, we have $\caE(\tilde{X})\subseteq \caE(X)$. In particular, $\caE(\caE(X))=\caE(X)$.
    
    \item \label{prop-part:compact} If $X$ is compact, then $\caE(X)$ is also compact.
    
    \item \label{prop-part:FX_EX_caEX} We have $X\subseteq\caF(X)\subseteq\caE(X)$, and $\caF(X)\preceq\TS(X)\preceq\caE(X).$
\end{enumerate}
\end{proposition}
\begin{proof} Part (\ref{prop-part:X subset EX}) is straightforward. For Part (\ref{prop-part:E^2=E}), we claim that for every $x\in X\subseteq\tilde{X}$, 
$\bigcup _{\tilde{y}\in\tilde{X}} I_{x\tilde{y}}= \bigcup _{y\in X} I_{xy}.$
Then it follows from the claim that 
\begin{align*}
   \caE(\tilde{X}) = \bigcap_{\tilde{x}\in \tilde{X}}\bigcup _{\tilde{y}\in\tilde{X}} I_{\tilde{x}\tilde{y}}\subseteq \bigcap_{x\in X}\bigcup _{\tilde{y}\in\tilde{X}} I_{x\tilde{y}}= \bigcap_{x\in X}\bigcup _{y\in X} I_{xy}=\caE(X).
\end{align*}
To prove the claim, it suffices to show $\bigcup _{\tilde{y}\in\tilde{X}} I_{x\tilde{y}}\subseteq\bigcup _{y\in X} I_{xy}$, because the reverse inclusion follows directly from the fact $X\subseteq\tilde{X}$. Fix $x\in X\subseteq\tilde{X}$. For each $z\in \bigcup _{\tilde{y}\in\tilde{X}} I_{x\tilde{y}}$, there exists $\tilde{y}\in \tilde{X}\subseteq\caE(X)$ such that $z\in I_{x\tilde{y}}$. Because $\tilde{y}\in 
\caE(X)$ and $x\in X$, there exists $y\in X$ such that $\tilde{y}\in I_{xy}$. It follows from item (\ref{prop:I_xy contains D_xz}) of Proposition \ref{prop:property of C_i and I_xy} that $z\in I_{x\tilde{y}}\subseteq I_{xy}$. Therefore,
\[\bigcup _{\tilde{y}\in\tilde{X}} I_{x\tilde{y}}\subseteq\bigcup _{y\in X} I_{xy},\]
and the claim is proved. Taking $\tilde{X}= \caE(X)$, we obtain $\caE(\caE(X))\subseteq \caE(X)$. Combined with Part (\ref{prop-part:X subset EX}), we then have $\caE(X)=\caE(\caE(X))$.

For Part (\ref{prop-part:compact}), we assume that $X$ is compact. By Theorem \ref{thm:compactness_tight_span} the tight span $\TS(X)$ is also compact, i.e. every limit of minimal functions is still minimal. Consider a convergent sequence $\{p_i\}_{i\in \mathbb{N}}\subseteq \caE(X)$. Then $d_\infty(\lim p_i,\cdot) = \lim d_\infty(p_i,\cdot)$ is the limit of minimal functions, and thus is also minimal. It follows that $\lim p_i\in \caE(X)$.

Finally, item (\ref{prop-part:FX_EX_caEX}) can be proved by using Proposition \ref{prop:FX_EX_caEX}, the minimality of $\TS(X)$, and the injectivity of $\caE(X)$ which will be shown in Proposition \ref{prop:inj_iff_Q=caEQ}.
\end{proof}

\begin{example}\label{ex:EX-caFX}  
The isometric embeddings in Proposition \ref{prop:property of E} (\ref{prop-part:FX_EX_caEX}) can all be non-surjective. 
For example, let $X:={(0,3)}\cup (\{0\}\times [-1,1])\subseteq \R^2_\infty$ as before. Then we have $X\subsetneq \caF(X)\prec \TS(X)\prec \caE(X)$, see Figure \ref{fig:caFX v.s. EX 2}. 

\begin{figure}[h!]
    \centering
    \begin{tikzpicture}[scale=0.8]
\node (a) at (-1,0) {};
\node (c) at (0,3) {};
\node (e) at (1,0) {};

\node (d1) at (4,1) {};

\node (a1) at (3,0) {};
\node (c1) at (4,3) {};
\node (e1) at (5,0) {};

\node (a2) at (7,0) {};
\node (c2) at (8,3) {};
\node (e2) at (9,0) {};
\node (d2) at (8,1) {};

\node (a3) at (12,0) {};
\node (c3) at (13,3) {};
\node (e3) at (14,0) {};
\node (b3) at (11,1) {};
\node (f3) at (15,1) {};

\filldraw (c) circle[radius=1pt];
\draw [thick] (e.center)--(a.center) node at (0,-0.5) {$X$}; 

\filldraw [color= green!80!blue] (c1) circle[radius=1pt];
\node[color= green!80!blue] at (4,-0.5) {$\caF(X)$}; 
\filldraw [color=green!80!blue] (a1.center) -- (d1.center) -- (e1.center) -- cycle;

\node[color= blue!80] at (8,-0.5) {$\TS(X)$}; 
\draw [color=blue!80] (d2.center)--(c2.center);
\filldraw [color=blue!50] (a2.center) -- (d2.center) -- (e2.center) -- cycle;

\filldraw [color=red!30] (a3.center) -- (b3.center) -- (c3.center) -- (f3.center) -- (e3.center) -- cycle;
\node [color=red] at (13,-0.5) {$\caE(X)$}; 
\end{tikzpicture} 
    \caption{An example such that $X\subsetneq \caF(X)\prec \TS(X)\prec \caE(X)$.}
    \label{fig:caFX v.s. EX 2}
\end{figure}
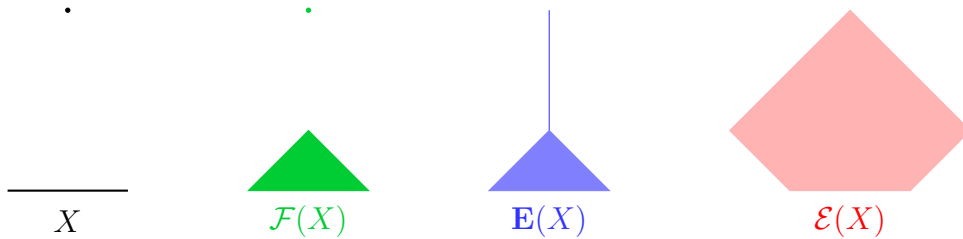
\end{example}

Next, we show that the set of $X$-minimal points is injective:

\begin{proposition}\label{prop:inj_iff_Q=caEQ}
For a compact $Q\subseteq\R^n_\infty$, $Q$ is injective if and only if $Q=\caE(Q)$. In addition, for any every compact set $X\subseteq\R^n_\infty$, $\caE(X)$ is injective.
\end{proposition}
\begin{proof}
Assume that $Q$ is injective. To show $Q=\caE(Q)$, it suffices to show that $Q\supseteq \caE(Q)$.
Since $Q$ is injective, $Q$ is isometric to $\TS(Q)$. As a consequence, any function $f\in\TS(Q)$ can be realized as a distance map $d_\infty^Q(z,\cdot):Q\to \R$ for some $z\in Q$, and thus has a zero. For any $p\in\caE(Q)$, the distance map $d_\infty^Q(p,\cdot):Q\to \R$ is in the tight span $\TS(Q)$. Because $d_\infty^Q(p,\cdot):Q\to \R$ has a zero, there is some $q\in Q$ such that $d_\infty^Q(p,q)=0$, i.e. $p=q\in Q.$ 

Conversely, suppose $Q=\caE(Q)$. Consider the map $\iota_Q:\R^n_\infty\to L^{\infty}(Q)\text{ with }z\mapsto d_\infty^Q(z,\cdot).$
The set $\iota_Q^{-1}(\TS(Q))$ consists of $z\in \R^n_\infty$ such that the function $d_\infty^Q(z,\cdot)$ are minimal, and thus is equal to $\caE(Q)(=Q)$ by Definition \ref{def:X-minimal}. Thus, $\iota_Q|_{Q}:Q\to \TS(Q)$ is surjective. In addition, because $\iota_Q|_{Q}$ preserves distance, $\iota_Q|_{Q}$ is an isometry. It follows from $Q\cong \TS(Q)$ that $Q$ is injective.

Recall from Proposition \ref{prop:property of E} (\ref{prop-part:E^2=E}) that $\caE(X)=\caE(\caE(X))$, for any every compact set $X\subseteq\R^n_\infty$. Thus, $\caE(X)$ is injective.
\end{proof}

\begin{remark}
For every compact set $X\subseteq\R^n_\infty$, $\TS(X)$ is isometric to the smallest compact set $Y\subseteq\R^n_\infty$ such that $X\subseteq Y\subseteq \caE(X)$ and $Y=\caE(Y)$.
\end{remark} 

\subsubsection{The tight span of $\Sp ^2_\infty$}
We have already seen in Theorem \ref{thm:ESn_infinity} that the tight span of $\Sp ^n_\infty$ (or that of $\bbD^{n+1}_\infty$) properly contains $\bbD^{n+1}_\infty$ for $n\geq 2$, up to isometry. Below, we concentrate to the case when $n=2$, and use the injectivity of $\caE(\Sp^2_\infty)$ to show that $\TS(\Sp^2_\infty)$ is isometric to $\caF(\Sp^2_\infty)=\caE(\Sp^2_\infty)$.

\begin{theorem} \label{thm:caE(S2)}
Let $\Sp ^2_\infty$ be the unit sphere centered at the origin in $\R^3_\infty$. Then, 

\begin{enumerate}
    \item \label{thm-part:boundary}
    $\caF(\Sp^2_\infty)=\caE(\Sp^2_\infty)$, which is then isometric to the tight span $\TS(\Sp^2_\infty)$ of $\Sp^2_\infty$.
    
    \item \label{thm-part:manifold}
    $\caF(\bbD^3_\infty)=\caE(\bbD^3_\infty)=\caF(\Sp^2_\infty)$, which is then isometric to the tight span $\TS(\bbD^3_\infty)$ of $\bbD^3_\infty$. In addition, $\TS(\bbD^3_\infty)\cong \TS(\Sp^2_\infty).$
\end{enumerate}
\end{theorem}

We prepare the following two lemmas for the proof of Theorem \ref{thm:caE(S2)}.

\begin{lemma}\label{prop:EX=caEX=caFX}
Let $X$ be a compact subset of $\R^n_\infty$.
\begin{enumerate}
    \item If the map $\iota_X:\caE(X)\to \TS(X)\text{ with }x\mapsto d_\infty^X(x,\cdot)$ preserves distances, then the tight span $\TS(X)\cong\caE(X)$.
    
    \item \label{prop-part:caE=caF}
    If $\caE(X)\subseteq\caF(X)$, i.e. $\caE(X)=\caF(X)$, then the tight span $\TS(X)\cong \caE(X)=\caF(X)$.
\end{enumerate}
\end{lemma}
\begin{proof} If the map $\iota_X:\caE(X)\to \TS(X)\text{ with }x\mapsto d_\infty^X(x,\cdot)$ preserves distance, then $\caE(X)$ is an injective metric space isometrically embedded into $\TS(X)$. By the minimality of $\TS(X)$, it must be true that $\caE(X)\cong \TS(X)$.

Assume now $\caE(X)=\caF(X)$. Recall from Proposition \ref{prop:FX_EX_caEX} that $\iota_X:\caF(X)=\caE(X)\to \TS(X)\text{ with }x\mapsto d_\infty^X(x,\cdot)$ preserves distances. Thus, $\TS(X)\cong \caF(X)=\caE(X)$.
\end{proof}

\begin{lemma} \label{lem:mirror intersection}
Let $X\subseteq \R^n_\infty$ and $p\in \caE(X)$. For each fixed $1\leq i\leq n$ and $\xi \in \{\pm 1\}$, if $\mathrm{Int}(p+\xi  \Lambda_i)\cap X\neq \emptyset$, then $(p-\xi  \Lambda_i)\cap X\neq \emptyset$.

\end{lemma}
\begin{proof}
Without loss of generality, assume that $i=1$ and $\xi=1$. Let $x\in \mathrm{Int}(p+  \Lambda_1)\cap X$. Since $p\in \caE(X)\subseteq \bigcup_{y\in X} I_{xy}$, there exists $y\in X$ such that $p\in I_{xy}$. Suppose $1\leq j\leq n$ and $\eta\in \{\pm 1\}$ are such that $y\in x+\eta  \Lambda_j$. Then $p\in I_{xy}=(x+\eta  \Lambda_j)\cap (y-\eta  \Lambda_j)\subseteq (x+\eta  \Lambda_j)$ by item (\ref{prop:I_xy decomp}) of Proposition \ref{prop:property of C_i and I_xy}. Combined with the fact that $p$ is in $\mathrm{Int}(x- \Lambda_1)$, we must have $\eta=-1$ and $j=1$. Thus, $p\in y-\eta  \Lambda_j=y+\Lambda_1$ and it follows that $y\in(p- \Lambda_1)\cap X$, i.e. $(p- \Lambda_1)\cap X\neq \emptyset$.
\end{proof}

\begin{proof}[Proof of Theorem \ref{thm:caE(S2)}]
To prove $\caF(\Sp^2_\infty)=\caE(\Sp^2_\infty)$, it is enough to show $\caE(\Sp^2_\infty)\subseteq\caF(\Sp^2_\infty)$. Then, item (\ref{thm-part:boundary}) immediately follows from Proposition \ref{prop:EX=caEX=caFX}.
Because the ball $\bbD^3_\infty$ is clearly contained in $\caF(\Sp^2_\infty)$, it suffices to show that $\caE(\Sp^2_\infty)-\bbD^3_\infty\subseteq \caF(\Sp^2_\infty)$. 

Take any $p\in \caE(\Sp^2_\infty)-\bbD^3_\infty$. Because $\Sp^2_\infty$ cannot be fully embedded into the union of finitely many $2$-dimensional planes $\bigcup_{i=1}^3\lbracket\partial(p+ \Lambda_i)\cup \partial(p- \Lambda_i)\rbracket $, there must exist some $1\leq i\leq 3$ and $\xi \in \{\pm 1\}$ such that $\mathrm{Int}(p+\xi  \Lambda_i)\cap \Sp^2_\infty\neq \emptyset.$ Without loss of generality, we assume that $\xi=1$ and $i=1$. By Lemma \ref{lem:mirror intersection}, we also have $(p-  \Lambda_1)\cap \Sp^2_\infty\neq \emptyset.$

If both $\mathrm{Int}(p+ \Lambda_i)\cap \Sp^2_\infty=\emptyset$ and $\mathrm{Int}(p- \Lambda_i)\cap \Sp^2_\infty=\emptyset$ for $i=2$ and $i=3$, then for every $x\in \Sp^2_\infty $, $x$ is either in $(p+ \Lambda_1)-\{p\}$ or $(p- \Lambda_1)-\{p\}$. Equivalently, we have either $x_1>p_1$ or $x_1<p_1$. Therefore,
\[\Sp^2_\infty \\
= \{x\in \Sp^2_\infty : x_1< p_1\} \sqcup \{x\in \Sp^2_\infty : x_1> p_1\}.\]

In addition, $\{x\in \Sp^2_\infty : x_1< p_1\}\supset (p+ \Lambda_1)\cap \Sp^2_\infty\neq \emptyset$ is nonempty, so is $\{x\in \Sp^2_\infty : x_1> p_1\}$. 
Thus, we have written $\Sp^2_\infty$ as the disjoint union of two nonempty open subsets, contradicting that $ \Sp^2_\infty $ is connected.
Hence, there must exist $2\leq i\leq 3$ such that $\mathrm{Int}(p+ \Lambda_i)\cap \Sp^2_\infty\neq\emptyset$ or $\mathrm{Int}(p- \Lambda_i)\cap \Sp^2_\infty\neq\emptyset$. Without loss of generality, suppose $\mathrm{Int}(p+ \Lambda_2)\cap \Sp^2_\infty\neq\emptyset$. By Lemma \ref{lem:mirror intersection}, we also have $(p-  \Lambda_2)\cap \Sp^2_\infty\neq \emptyset.$

If both $\mathrm{Int}(p+ \Lambda_3)\cap \Sp^2_\infty=\emptyset$ and $\mathrm{Int}(p- \Lambda_3)\cap \Sp^2_\infty=\emptyset$, then
\[\Sp^2_\infty \subseteq (p+ \Lambda_1)\cup (p- \Lambda_1)\cup (p+ \Lambda_2)\cup (p- \Lambda_2) -\{p\}.\]
Denote the regions $\mathrm{Int}(p+ \Lambda_1)\cap \Sp^2_\infty, \mathrm{Int}(p+ \Lambda_2)\cap \Sp^2_\infty,(p- \Lambda_1)\cap \Sp^2_\infty$ and $ (p- \Lambda_2)\cap \Sp^2_\infty$ by $R^1,R^2,R^3$ and $R^4$ respectively. For $j=1,2,3,4$, take $x^j\in R^j$, and let $\gamma_j\subseteq R_j\cup R_{j+1}$ be a path connecting $x^j$ and $x^{j+1}$, where we assume $R_5=R_1$ and $x^5=x^1$. Because $\Sp^2_\infty$ is simply-connected, the loop $\gamma:=\gamma_1\circ\gamma_2\circ\gamma_3\circ\gamma_4$ is contractible in $\Sp ^2$, and thus contractible in $(p+ \Lambda_1)\cup (p- \Lambda_1)\cup (p+ \Lambda_2)\cup (p- \Lambda_2) -\{p\}$. This gives a contradiction.
\end{proof}


\newcommand{\etalchar}[1]{$^{#1}$}


\begin{thebibliography}{DHK{\etalchar{+}}12}
	
	\bibitem[AA17]{adamaszek2017vietoris}
	Micha{\l} Adamaszek and Henry Adams.
	\newblock The {V}ietoris--{R}ips complexes of a circle.
	\newblock {\em Pacific Journal of Mathematics}, 290(1):1--40, 2017.
	
	\bibitem[AAF18]{adamaszek2018metric}
	Micha{\l} Adamaszek, Henry Adams, and Florian Frick.
	\newblock Metric reconstruction via optimal transport.
	\newblock {\em SIAM Journal on Applied Algebra and Geometry}, 2(4):597--619,
	2018.
	
	\bibitem[AP56]{aronszajn1956extension}
	Nachman Aronszajn and Prom Panitchpakdi.
	\newblock Extension of uniformly continuous transformations and hyperconvex
	metric spaces.
	\newblock {\em Pacific Journal of Mathematics}, 6(3):405--439, 1956.
	
	\bibitem[BBI01]{burago2001course}
	Dmitri Burago, Yuri Burago, and Sergei Ivanov.
	\newblock {\em A course in metric geometry}, volume~33.
	\newblock American Mathematical Soc., 2001.
	
	\bibitem[Car09]{carlsson2009topology}
	Gunnar Carlsson.
	\newblock Topology and data.
	\newblock {\em Bulletin of the American Mathematical Society}, 46(2):255--308,
	2009.
	
	\bibitem[CCR13]{chan2013topology}
	Joseph~Minhow Chan, Gunnar Carlsson, and Raul Rabadan.
	\newblock Topology of viral evolution.
	\newblock {\em Proceedings of the National Academy of Sciences},
	110(46):18566--18571, 2013.
	
	\bibitem[Cha72]{chapman1972structure}
	Thomas~A. Chapman.
	\newblock On the structure of {H}ilbert cube manifolds.
	\newblock {\em Compositio Mathematica}, 24(3):329--353, 1972.
	
	\bibitem[CS74]{curtis1974}
	Douglas~W Curtis and Richard~M Schori.
	\newblock $2^x $ and $c(x) $ are homeomorphic to the {H}ilbert cube.
	\newblock {\em Bulletin of the American Mathematical Society}, 80(5):927--931,
	1974.
	
	\bibitem[DD09]{deza2009encyclopedia}
	Michel~Marie Deza and Elena Deza.
	\newblock Encyclopedia of distances.
	\newblock In {\em Encyclopedia of distances}, pages 1--583. Springer, 2009.
	
	\bibitem[DHK{\etalchar{+}}12]{dress2012basic}
	Andreas Dress, Katharina~T Huber, Jacobus Koolen, Vincent Moulton, and Andreas
	Spillner.
	\newblock {\em Basic phylogenetic combinatorics}.
	\newblock Cambridge University Press, 2012.
	
	\bibitem[DP17]{descombes2017injective}
	Dominic Descombes and Ma{\"e}l Pav{\'o}n.
	\newblock Injective subsets of $l_{\infty}(i)$.
	\newblock {\em Advances in Mathematics}, 317:91--107, 2017.
	
	\bibitem[Dre84]{dress1984trees}
	Andreas Dress.
	\newblock Trees, tight extensions of metric spaces, and the cohomological
	dimension of certain groups: a note on combinatorial properties of metric
	spaces.
	\newblock {\em Advances in Mathematics}, 53(3):321--402, 1984.
	
	\bibitem[Epp11]{eppstein2011optimally}
	David Eppstein.
	\newblock Optimally fast incremental {M}anhattan plane embedding and planar
	tight span construction.
	\newblock {\em Journal of Computational Geometry}, 2(1), 2011.
	
	\bibitem[GM00]{goodman2000tight}
	Oliver Goodman and Vincent Moulton.
	\newblock On the tight span of an antipodal graph.
	\newblock {\em Discrete Mathematics}, 218(1-3):73--96, 2000.
	
	\bibitem[Her92]{herrlich1992hyperconvex}
	Horst Herrlich.
	\newblock Hyperconvex hulls of metric spaces.
	\newblock {\em Topology and its Applications}, 44(1-3):181--187, 1992.
	
	\bibitem[Isb64]{isbell1964six}
	John~R Isbell.
	\newblock Six theorems about injective metric spaces.
	\newblock {\em Commentarii Mathematici Helvetici}, 39(1):65--76, 1964.
	
	\bibitem[JJ19]{joharinad2019topology}
	Parvaneh Joharinad and J{\"u}rgen Jost.
	\newblock Topology and curvature of metric spaces.
	\newblock {\em Advances in Mathematics}, 356:106813, 2019.
	
	\bibitem[JJ20]{joharinad2020topological}
	Parvaneh Joharinad and J{\"u}rgen Jost.
	\newblock Topological representation of the geometry of metric spaces.
	\newblock {\em arXiv preprint arXiv:2001.10262}, 2020.
	
	\bibitem[Kat83]{katz1983filling}
	Mikhail Katz.
	\newblock The filling radius of two-point homogeneous spaces.
	\newblock {\em Journal of Differential Geometry}, 18(3):505--511, 1983.
	
	\bibitem[Kat89]{katz1989diameter}
	Mikhail Katz.
	\newblock Diameter-extremal subsets of spheres.
	\newblock {\em Discrete \& Computational Geometry}, 4(2):117--137, 1989.
	
	\bibitem[Kat91]{katz1991neighborhoods}
	Mikhail Katz.
	\newblock On neighborhoods of the {K}uratowski imbedding beyond the first
	extremum of the diameter functional.
	\newblock {\em Fundamenta Mathematicae}, 137(3):161--175, 1991.
	
	\bibitem[KK15]{kilicc2015tight}
	Mehmet Kili{\c{c}} and {\c{S}}ahin Ko{\c{c}}ak.
	\newblock Tight span of path connected subsets of the {M}anhattan plane.
	\newblock {\em arXiv preprint arXiv:1507.05041}, 2015.
	
	\bibitem[KK16]{kilicc2016tight}
	Mehmet K{\i}l{\i}{\c{c}} and {\c{S}}ahin Ko{\c{c}}ak.
	\newblock Tight span of subsets of the plane with the maximum metric.
	\newblock {\em Advances in Mathematics}, 301:693--710, 2016.
	
	\bibitem[KK{\"O}21]{kilicc2021algorithm}
	Mehmet K{\i}l{\i}{\c{c}}, {\c{S}}ahin Ko{\c{c}}ak, and Yunus {\"O}zdemir.
	\newblock An algorithm for the construction of the tight span of finite subsets
	of the {M}anhattan plane.
	\newblock {\em Computational Geometry}, 95:101741, 2021.
	
	\bibitem[Kle55]{klee1955some}
	Victor~L Klee.
	\newblock Some topological properties of convex sets.
	\newblock {\em Transactions of the American Mathematical Society},
	78(1):30--45, 1955.
	
	\bibitem[Lan13]{lang2013injective}
	Urs Lang.
	\newblock Injective hulls of certain discrete metric spaces and groups.
	\newblock {\em Journal of Topology and Analysis}, 5(03):297--331, 2013.
	
	\bibitem[LMO20]{lim2020vietoris}
	Sunhyuk Lim, Facundo Memoli, and Osman~Berat Okutan.
	\newblock Vietoris-{R}ips persistent homology, injective metric spaces, and the
	filling radius.
	\newblock {\em arXiv preprint arXiv:2001.07588}, 2020.
	
	\bibitem[LPZ13]{lang2013metric}
	Urs Lang, Ma{\"e}l Pav{\'o}n, and Roger Z{\"u}st.
	\newblock Metric stability of trees and tight spans.
	\newblock {\em Archiv der Mathematik}, 101(1):91--100, 2013.
	
	\bibitem[MOW18]{memoli2018metric}
	Facundo M{\'e}moli, Osman~Berat Okutan, and Qingsong Wang.
	\newblock Metric graph approximations of geodesic spaces.
	\newblock {\em arXiv preprint arXiv:1809.05566}, 2018.
	
	\bibitem[Pav16]{pavon2016injective}
	Ma{\"e}l Pav{\'o}n.
	\newblock Injective convex polyhedra.
	\newblock {\em Discrete \& Computational Geometry}, 56(3):592--630, 2016.
	
	\bibitem[Rol84]{rolewicz1984optimal}
	Stefan Rolewicz.
	\newblock On optimal observability of {L}ipschitz systems.
	\newblock In {\em Selected topics in operations research and mathematical
		economics}, pages 152--158. Springer, 1984.
	
	\bibitem[SW74]{schori1974hyperspaces}
	Richard Schori and James West.
	\newblock Hyperspaces of graphs are {H}ilbert cubes.
	\newblock {\em Pacific Journal of Mathematics}, 53(1):239--251, 1974.
	
	\bibitem[SW75]{schori1975hyperspace}
	Richard Schori and James West.
	\newblock The hyperspace of the closed unit interval is a {H}ilbert cube.
	\newblock {\em Transactions of the American Mathematical Society},
	213:217--235, 1975.
	
\end{thebibliography}
\end{document}